%% file: master.tex
\documentclass[10pt]{amsart}
\input{macro}

\author[Moreno Del Angel]{Juan C. Moreno Del Angel}

\title[Duals of higher real $K$-theories]{Duals of higher real $K$-theories at $p=2$}
\begin{document}

\begin{abstract}
We study $\mK(h)$-local Spanier-Whitehead duality for $C_{2^n}$-equivariant Lubin-Tate spectra, 
$E_h$, at the prime $2$ and heights $h$ divisible by $2^{n-1}$. We determine a $C_{2^n}$-equivariant
equivalence $DE_h\simeq\Sigma^{-V_h} E_h$, for an explicit
$C_{2^n}$-representation, $V_h$.
We then study the $\mathrm{RO}(C_{2^n})$-periodicities of 
$E_h$ at some low heights. With these ingredients, we determine the self-duality of some higher 
real $K$-theories up to a specified suspension shift, at some low-heights. In particular,
we show that $DE_4^{hC_8}\simeq \Sigma^{112}E_4^{hC_8}$. 
\end{abstract}
\maketitle
\setcounter{tocdepth}{1}
\tableofcontents

\input{Section1-Intro.tex}

\input{Section2-EqDual.tex}

\input{Section3-DualRepGen.tex}

\input{rblt.tex}
\input{Section4-Orientability.tex}
\input{Section5-IntShifts.tex}

\appendix
\input{AppendixA-nonequivariantdual.tex}
\bibliographystyle{amsalpha}
\bibliography{mduals}
%\printbibliography
\end{document}

%% file: macro.tex
\usepackage{amsmath}
\usepackage{amsfonts}
\usepackage{amssymb}
\usepackage{amsthm,thmtools}
\usepackage{bm}
\usepackage{comment}
\usepackage{epsfig}
\usepackage{psfrag}
\usepackage{mathrsfs}
\usepackage{amscd}
\usepackage{rotating}
\usepackage{lscape}
\usepackage{amsbsy}
\usepackage{verbatim}
\usepackage{moreverb}
\usepackage{sseq}
\usepackage{pdfpages}
\usepackage{enumerate}
\usepackage{float}
\usepackage{stmaryrd}
\restylefloat{table}
\usepackage{tikz-cd}
\usetikzlibrary{decorations.pathmorphing}
\usepackage{url}
\usepackage[utf8]{inputenc}
\usepackage[color = teal!40]{todonotes}
\usepackage{graphicx}
\usepackage{color}
\usepackage[colorlinks]{hyperref}

\def\tmp#1#2#3{%
  \definecolor{Hy#1color}{#2}{#3}%
  \hypersetup{#1color=Hy#1color}}
\tmp{file}{HTML}{800006}
\tmp{url}{HTML}{2E7E2A}
\tmp{link}{HTML}{131877}
\tmp{cite} {HTML}{8A0087}
\tmp{menu}{HTML}{727500}
\tmp{run} {HTML}{137776}
\def\tmp#1#2{%
  \colorlet{Hy#1bordercolor}{Hy#1color#2}%
  \hypersetup{#1bordercolor=Hy#1bordercolor}}
\tmp{link}{!60!white}
\tmp{cite}{!60!white}
\tmp{file}{!60!white}
\tmp{url} {!60!white}
\tmp{menu}{!60!white}
\tmp{run} {!60!white}

\usepackage[capitalise]{cleveref}
\numberwithin{equation}{subsection}

\newcommand{\mmod}[1]{ \mathrm{mod}\ #1}

\theoremstyle{plain}
\newtheorem{thm}{Theorem}[subsection]
\newtheorem{cor}[thm]{Corollary}

\newtheorem{prop}[thm]{Proposition}
\newtheorem{lem}[thm]{Lemma}

\theoremstyle{definition}

\newtheorem{defn}[thm]{Definition}
\newtheorem{rmk}[thm]{Remark}

\newtheorem{quest}{Question}
\newtheorem{ass}{Assumption}

\newtheorem{alphthm}{Theorem}

\newtheorem{alphcor}{Corollary}

\newtheorem{alphquest}{Question}

\theoremstyle{plain}
\newtheorem{sthm}{Theorem}[section]

\newtheorem{slem}[sthm]{Lemma}

\theoremstyle{definition}
\newtheorem{sdefn}[sthm]{Definition}
\newtheorem{srmk}[sthm]{Remark}
\newtheorem{sex}[sthm]{Example}

\newtheorem*{thm*}{Theorem}
\newtheorem*{cor*}{Corollary}
\newtheorem*{conj*}{Conjecture}
\newtheorem*{prop*}{Proposition}
\newtheorem*{lem*}{Lemma}
\newtheorem*{pred*}{Prediction}
\newtheorem*{defn*}{Definition}
\newtheorem*{guess*}{Guess}
\newtheorem*{fact*}{Fact}
\newtheorem*{notation*}{Notation}
\newtheorem*{note*}{Note}
\newtheorem*{ex*}{Example}
\newtheorem*{rmk*}{Remark}
\newtheorem*{warn*}{Warning}
\newtheorem*{claim*}{Claim}
\newtheorem*{goal*}{Goal}
\newtheorem*{convention*}{Convention}
\newtheorem*{quest*}{Question}
\newtheorem*{ass*}{Assumptions}

\newcommand{\mK}{\mathrm{K}}

\newcommand{\stab}{ {\mathcal{O}^\times}}
\newcommand{\estab}{\mathbb{G}}

\newcommand{\mur}{\mathrm{MU}_\mathbb{R}}
\newcommand{\bpr}{\mathrm{BP}_\R}
\newcommand{\BPCfour}{\mathrm{BP}^{( (C_4) )}}
\newcommand{\MUCn}{\mathrm{MU}^{( (C_{2^n}) )}}
\newcommand{\BPCn}{\mathrm{BP}^{( (C_{2^n}) )}}
\newcommand{\ECnm}{E_{C_{2^n}}(m)}
\newcommand{\bpcfourone}{\mathrm{BP}^{( (C_4) )}\langle 1\rangle}
\newcommand{\bpcfourtwo}{\mathrm{BP}^{( (C_4) )}\langle 2\rangle}

\newcommand{\tgen}{\overline{t}}

\DeclareMathOperator{\spectra}{Sp}
\DeclareMathOperator{\cofib}{cofib}

\DeclareMathOperator{\slicess}{SliceSS}
\DeclareMathOperator{\hfpss}{HFPSS}
\DeclareMathOperator{\tatess}{TateSS}

\DeclareMathOperator{\res}{res}
\DeclareMathOperator{\trans}{tr}

\makeatletter
\let\c@equation=\c@thm
\let\c@lem=\c@thm
\let\c@cor=\c@thm
\let\c@conj=\c@thm
\let\c@prop=\c@thm
\let\c@lem=\c@thm
\let\c@defn=\c@thm
\let\c@notation=\c@thm
\let\c@note=\c@thm
\let\c@exmp=\c@thm
\let\c@ex=\c@thm
\let\c@exmps=\c@thm
\let\c@rem=\c@thm
\let\c@warn=\c@thm
\let\c@claim=\c@thm
\let\c@convention=\c@thm
\let\c@conventions=\c@thm
\let\c@quest=\c@thm
\let\c@thmintro=\c@thm
\let\c@conjintro=\c@thm
\let\c@thmbig=\c@thm
\let\c@conbig=\c@thm
\let\c@facts=\c@thm
\let\c@definition=\c@thm
\let\c@proposition=\c@thm
\let\c@lemma=\c@thm
\let\c@remark=\c@thm
\let\c@conjecture=\c@thm
\let\c@theorem=\c@thm
\let\c@prediction=\c@thm
\let\c@guess=\c@thm
\makeatother

\newcommand{\N}{\mathbb{N}}
\newcommand{\ox}{\otimes}

\DeclareMathOperator{\Mod}{Mod}
\DeclareMathOperator{\map}{map}

\DeclareMathOperator{\End}{End}

\DeclareMathOperator{\Gal}{Gal}
\newcommand{\smsh}{\wedge}

\DeclareMathOperator{\Aut}{Aut}
\DeclareMathOperator*\colim{colim}

\newcommand{\Z}{\mathbb{Z}}

\newcommand{\R}{\mathbb{R}}
\newcommand{\F}{\mathbb{F}}
\newcommand{\E}{\mathbb{E}}

\newcommand{\sphere}{\mathbb{S}}

\renewcommand{\to}{\longrightarrow}

\newcommand{\lieg}{\mathfrak{g}}

\newcommand{\cH}{\mathcal{H}}

\newcommand{\cG}{\mathcal{G}}

\newcommand{\cO}{\mathcal{O}}

%% file: Section1-Intro.tex
% !TEX root = master.tex

\section{Introduction}\label{intro}

The purpose of this article is to study duality for certain higher chromatic analogs of 
topological $K$-theories. The story of $K$-theory is now classical and has resulted in 
many applications ranging from Adams' solution to the vector fields on spheres problem \cite{AdamsVectorfieldsonspheres}, 
to the classification of free fermion phases in condensed-matter physics \cite{KitaevPeriodicTable}. 
Climbing up one wrung in the chromatic ladder leads to theories like topological modular forms, which 
detects complex patterns in the stable homotopy groups of spheres \cite{BehrensMahowaldQuigley,BelmontShimomura},
and has a conjectural geometric description in terms of 2-dimensional supersymmetric field theories \cite{stolzteichner}. 
The higher versions of real $K$-theory that we will study are at the moment somewhat mysterious,
though steady progress is being made towards understanding them. For some examples of recent work, see
\cite{bbhs,picardHLS,PrasitHoodeo,CarrickWoodDefect,HoodHuMorgan}.

The duality we are interested in is $\mK(h)$-local Spanier-Whitehead duality. Working at the 
prime $2$ and height $1$, Hahn-Mitchell identify $2$-completed $\mathrm{KO}$ as self-dual up 
to a desuspenson \cite{HahnMitchellDKO}. Their proof makes use of the Adams-Baird-Ravenel fiber sequence 
\cite{Adams1974,Ravenel1984}, which consequently takes the form
\begin{equation*}
    \label{KOfiberseq}
    D(\mathrm{KO}_2)\to L_{\mK(1)}\sphere\to \mathrm{KO}_2.
\end{equation*}
The takeaway is that every element in the $\mK(1)$-local part of the stable homotopy groups of spheres is 
either detected by the unit map to $\mathrm{KO}_2$ or its dual. 

The analogous picture at height $2$ has recently been made clear thanks to much hard work.
In this case, one has a spectral sequence which collapses at a finite page and
computes the homotopy groups of a spectrum closely related to $L_{\mK(2)}\sphere$ \cite{ghmrResolution,bobkovagoerss,finiteresolutions}. 
The $E_1$-page is given by homotopy groups of variants of $\mathrm{TMF}$, and
the edge homomorphism is the dual of the unit map of $L_{\mK(2)}\mathrm{TMF}$. 
The dual in this case was identified by Behrens at the prime $3$ \cite{BehrensModularK2} and 
Bobkova at the prime $2$ \cite{Bobkova_dual}. 
In general, the theory suggests that the higher versions of real $K$-theory and their 
duals form building blocks for $L_{\mK(h)}\sphere$. 

We now introduce the higher versions of real $K$-theory. Fix a perfect field $k$ of characteristic 
$p>0$ and a formal group law $\Gamma$ of height $h$ over $k$. 
For every pair $(k,\Gamma)$, there is a $\mK(h)$-local commutative ring spectrum, 
the Lubin-Tate spectrum, $E_h$, which comes equipped with an action by the
(extended) Morava stabilizer group, $\estab_h = \Aut(k,\Gamma)$ 
\cite{rezkHopkinsMiller, Goerss_Hopkins_2004, LurieEllipticII}. We let $\mathcal{O}^\times_h$ 
denote the subgroup of $\estab_h$ consisting of automorphisms leaving $k$ fixed, also 
called the (small) stabilizer group.
For finite subgroups $G\leq\estab_h$, one can 
form the homotopy fixed point spectra, $E_h^{hG}$. These are the \emph{higher real $K$-theories}. 

We will work almost exclusively at the prime $2$ and heights divisible by $2^{n-1}$. 
Under these assumptions there is a cyclic group of order $2^n$, 
$C_{2^n}\leq\mathcal{O}^\times_h$\cite{HewettFiniteSubs}, so that $E_h$ has an action by $C_{2^n}$. 
This $C_{2^n}$-spectrum is our central object of interest. 
At height $1$, we recover the classical real $K$-theory,
so by \cite{HahnMitchellDKO}, \[DE_1^{hC_2}\simeq 
\Sigma^{-1}E_1^{hC_2}.\] 
At height $2$, $E_2^{hC_4}$ is related to a version of $\mathrm{TMF}$ with 
level structure, $\mathrm{TMF}_0(5)$. Then from Bobkova's Theorem $1$ in \cite{Bobkova_dual}, we can deduce that
\[DE_2^{hC_4}\simeq\Sigma^{44}E_2^{hC_4}.\]
One of our main results (\cref{De4hC8}) identifies the dual of the next higher real $K$-theory in 
this pattern, $E_4^{hC_8}$, which is closely related to the detecting spectrum, $\Omega$, of 
Hill-Hopkins-Ravenel \cite{hhr16}. Before stating it, we describe our method.

For any finite $G\leq\estab_h$, work of Beaudry-Goerss-Hopkins-Stojanoska \cite{dualizingspheres} 
together with a theorem of Clausen \cite{ClausenAduality}
allows us to identify the $G$-equivariant dual of $E_h$ as a suspension
by a $G$-representation sphere. Explicitly, if  $V_h$ denotes the $p$-adic Lie algebra of $\estab_h$ 
viewed as a $G$-representation via the adjoint action, then there is a $G$-equivariant equivalence
\[DE_h\simeq_G \Sigma^{-V_h}E_h.\]
In \cref{duality}
we review some of the theory laid out in \cite{dualizingspheres} and in \cref{therep} 
we identify the representation sphere in the case $h$ is divisible by $2^{n-1}$ and $G=C_{2^n}$. 
This leads to the first main result of this paper, 
which was known to the authors of \cite{dualizingspheres}, but not published.
The author is grateful to them for their permission to record it here.
See \cref{thedualizingrep} for a more explicit statement.

\begin{alphthm}[\cref{thedualizingrep}]
\label{therepthm}
Let $\estab_h$ be the Morava stabilizer group at the prime $2$ and height $h=2^{n-1}m$, for $m$ odd. 
Let $V_h$ be the $p$-adic lie algebra of $\estab_h$, viewed as 
a $C_{2^n}$-representation by the adjoint action of $C_{2^n}\leq\mathcal{O}^\times_h\leq \estab_h$. 
Then $V_h$ is isomorphic to a direct sum of $h^2/2^{n-1}$ copies of the regular 
representation of $C_{2^n}/C_2\cong C_{2^{n-1}}$, viewed as a $C_{2^n}$-representation 
by restriction along the quotient map.
\end{alphthm}

%\begin{alphcor}
%    \label{thedualtheorem}
%    Let $E_h$ be a good Lutin-Tate spectrum at height $h = 2^{n-1}m$, for $m$ odd, and let $\mathcal{O}^\times_h$ 
%    be the small stabilizer group. There is a $C_{2^n}$-equivariant equivalence
%    \[DE_h\simeq_{C_{2^n}} \Sigma^{-V_h} E_h,\]
%    where $C_{2^n}\leq\mathcal{O}^\times_h$ acts diagonally on the right, and 
%    $V_h$ is the $C_{2^n}$-representation of \cref{therepthm}. 
%\end{alphcor}
\setcounter{alphcor}{1}

Interestingly, one is often able to identify $E_h$ as equivariantly self-dual up to an integer suspension 
(See \cref{fig:shiftstable}). This is nice because it implies  
that the self-duality will descend to the higher real $K$-theory upon taking homotopy fixed points. 
As far as the author knows, there is no counterexample to this statement.
This leads us to the notion of an equivariant periodicity.
For any $G$-spectrum, $X$, we say $X$ is \emph{V-periodic} or that $V$ is an 
\emph{$\mathrm{RO}(G)$-periodicity} of $X$ if there is a $G$-equivariant 
equivalence $\Sigma^V X\simeq_G X$. 
Given \cref{therepthm}, the self-duality of a higher real $K$-theory equivalent to the 
existence of a $(V_h+s_h)$-periodicity of $E_h$.
\begin{alphquest}\label{intshiftq}
    Is there an $s_h\in\Z$ such that the $G$-spectrum $E_h$ is $(V_h+s_h)$-periodic?
\end{alphquest}
We review some theory around periodicities of equivariant ring spectra in \cref{orientability} 
along with some equivariant computational methods that can be used to study \cref{intshiftq} at some 
low heights. Our approach differs from that of \cite{dualizingspheres} in that there, the authors
make use of known vanishing lines on the $E_\infty$-page of the multipliciative homotopy 
fixed point spectral sequence computing the $\mathrm{RO}(G)$-image in $\mathrm{Pic}(E^{hG})$. 
We instead use some advancements in computational techniques \cite{hhr16, hhr17, hishiwax, bbhs} 
and the computationally convenient models of Lubin-Tate spectra from \cite{lubintatemodels}
to produce explicit classes realizing equivalences of the form $\Sigma^{|V|-V}E
\simeq_G E$ by multiplication. By choosing $V$ strategically, we avoid computing 
the full $\text{RO}(G)$-image in the Picard group while still studying \cref{intshiftq}. 
This leads to the other main results of this paper. The first of these follows from some
known periodicities that are consequences of the work of Hill-Hopkins-Ravenel in \cite{hhr16,hhr17}. 
See \cref{LTcond} for what we mean by a \emph{good} Lubin-Tate spectrum. 

\begin{alphthm}[{\cref{C4shifts}}]
    \label{theC4shiftstheorem}
    Let $E_h$ be a good Lubin-Tate spectrum at the prime $2$ and height $h$. Then for $C_4\leq\mathcal{O}^\times_h$, there are 
    $C_4$-equivalences
     \[DE_h\simeq_{C_4} \begin{cases} \Sigma^{12}E_h & h = 2\\
    \Sigma^{-h^2}E_h & h = 4,8\end{cases}.\] 
\end{alphthm}

\begin{alphcor}[{\cref{C4shiftshfpcor}}]
    If $G\leq\estab_h$ is a finite subgroup with $G\cap\mathcal{O}^\times_h\leq C_4$, then we have 
    equivalences \[DE_h^{hG}\simeq \begin{cases} \Sigma^{12}E_h^{hG} & h = 2\\
    \Sigma^{-h^2}E_h^{hG} & h = 4,8\end{cases}.\]  
\end{alphcor}

\begin{srmk}
    We note that the equivalence for $h=2$ follows from 
    \cite[Theorem~1]{Bobkova_dual}, and is also recovered in \cite{dualizingspheres}. 
\end{srmk}

\begin{alphthm}[\cref{height4p,C8shift}]
\label{theC8shifttheorem}
    Let $E_4$ be a good Lubin-Tate spectrum at the prime $2$ and height $h=4$. Then for $C_8\leq\mathcal{O}^\times_4$,
    $E_4$ is $V$-periodic for the virtual $C_8$-representation $V$ given in \cref{height4p}.
    Consequently, there is a $C_8$-equivalence\[DE_4\simeq_{C_8} \Sigma^{112}E_4.\]
\end{alphthm}

\begin{alphcor}[{\cref{C8shiftshfpcor}}]\label{De4hC8}
    Let $E_4$ be a good Lubin-Tate spectrum and $G\leq\estab_4$ a finite subgroup 
    with $G\cap\mathcal{O}^\times_4 = C_8$. Then 
    \[DE_4^{hG}\simeq \Sigma^{112}E_4^{hG}.\]
\end{alphcor}

\begin{figure}[H]
    \begin{center}
        \begin{tabular}{| c | c | c | c | c |}
        \hline
        $G$ & $h$ & $s_h$ &  Ref. \\
        \hline
        \hline
        $C_2$ & $1$ & $-1$ & \cite{HahnMitchellDKO}\\
        \hline
        $C_2$ & $\geq 1$ & $-h^2$ & \cite{dualizingspheres}\\
        \hline
        $C_4$ & $2$ & $12$ & \cite{Bobkova_dual}\\
        \hline
        $C_4$ & $8$ & -64 & \cref{C4shifts}\\
        \hline
        $C_8$ & $4$ & $112$ & \cref{C8shift}\\
        \hline
        \end{tabular}
        \quad 
        \begin{tabular}{| c | c | c | c | c |}
        \hline
        $G$ & $h$ & $s_h$ & Ref. \\
        \hline
        \hline
        $C_3$ & 2 & $-20$ & \cite{BehrensModularK2}\\
        \hline
        $C_p$ & $p-1$ & $-h^2(2p-1)$ & \cite{dualizingspheres}\\
        \hline
        \end{tabular}
        \caption{Examples where $E_h$ at $p=2$ (left) and odd $p$ (right) satisfies
        $DE_h\simeq_G\Sigma^{s_h}E_h$.}
        \label{fig:shiftstable}
    \end{center}
    \end{figure}
    \vspace{-1em}

We end this introduction by remarking that the `algebraically closed' Lubin-Tate spectra,
i.e. those associated to a formal group law over $\overline{\F}_p$ are not 
good Lubin-Tate spectra for the theorems stated above, so the duality theorems do not 
necessarily apply to them. We include an appendix studying this case
nonequivariantly. Therein we compute the homotopy groups of the dual of an
algebraically closed Lubin-Tate spectrum, from which we can conclude that indeed the 
theorems above cannot hold as is $E$ algebraically closed. 

\begin{alphthm}[{\cref{hmtpyClosedEthy}}]
Let $E(\overline{\F}_p)$  denote the algebraically closed Lubin-Tate spectrum associated to the Honda 
formal group law.
 Then the homotopy groups of $DE(\overline{\F}_p)$ are given by 
    \[\pi_\ast DE(\overline{\F}_p)\cong E(\F_{p^h})_{\ast+h^2}\hat{\ox}_{W(\F_{p^h})}D^c_h W(\overline{\F}_p).\]
\end{alphthm}

\subsection*{Notation \& conventions}
We will use $\rho_{2^n}$ and $\sigma_{2^n}$ to denote the real regular representation and 
$1$-dimensional sign 
representation of $C_{2^n}$, the cyclic group of order $2^n$, dropping the subscripts 
when there is no ambiguity. 
We will write $\lambda_i$ for  the $2$-dimensional representation of $C_{2^n}$
given by rotation by $\frac{2\pi}{2^{n-i}}$ for $0\leq i < n-1$, leaving the group implicit. 
We will also write $\lambda = \lambda_0$ and take the convention that $\lambda_{n-1} = 2\sigma$. 
If $V$ is an irreducible representation then after $2$-completion $S^V\simeq_{C_{2^n}}S^W$ for 
$W$ either $1$, $\sigma$, or $\lambda_i$ for some $0\leq i < n-1$. 
Since this is the relevant notion of equivalence of representations for us, we will write 
\[\rho_{2^n} = 1 + \sigma + \sum_{i=0}^{n-2}2^{n-2-i}\lambda_i.\]

Let $G$ be a finite group.
We let $\spectra^G$ denote the
$\infty$-category of genuine $G$-spectra obtained by taking the homotopy coherent nerve of 
bifibrant objects in the stable model structure on the category of orthogonal $G$-spectra \cite{Mandell2002}.
Our main objects of interest are spectra with $G$-action, that is, objects of the 
$\infty$-category $\spectra^{BG}$. We view these as a geniuine $G$-spectra using the 
\textit{cofree extension of universe} $\spectra^{BG}\to\spectra^G$.
This will be left implicit in the notation.  For $X\in\spectra^G$
we will write $X^h = F(EG_+,X)\in\spectra^G$ for the function spectrum.
With our conventions, if $X$ is a spectrum with $G$-action 
then $X^h\simeq_G X$ so that any such $X$ is \textit{cofree}.

\subsection*{Some spectral sequences}
For $X\in\spectra^G$, we write $P^n X$ to denote 
the $n$-th slice coconnected cover in the sense of \cite{UllmanRSSS} and 
$P^n_n X$ for the $n$-th slice section, that is, the fiber of the map 
$P^n X\to P^{n-1}X$. The resulting spectral sequence on homotopy groups has signature 
\[E_2^{s,V}=\pi_{V-s}^G P^{|V|}_{|V|} X \implies \pi_{V-s}^G X,\] for $V\in\textrm{RO}(G)$.
The differentials on the $E_r$-page are of the form 
\[d_r:E_2^{s,V}\to E_2^{s+r,V+r-1}\] 
This is the \textit{$G$-slice spectral sequence} of $X$ and we will denote it by 
$\slicess(X)$.

Let $\widetilde{EG} = \cofib(EG_+\to S^0)$. 
We will also consider the towers obtained by applying $F(EG_+, -)$ and 
$\widetilde{EG}\smsh F(EG_+, -)$ to the slice tower. The spectral sequence 
associated to the first of these has signature 
\[E_2^{s,V}=\pi_{V-s}^G F(EG_+,P^{|V|}_{|V|}X)\implies \pi_{-s}(\Sigma^{-V}X)^{hG}.\]
This is the \textit{homotopy fixed point spectral sequence} of $X$ and we will denote it by 
$\hfpss(X)$. The spectral sequence associated to the second of these has signature
\[E_2^{s,V} = \pi_{V-s}^G \widetilde{EG}\smsh F(EG_+,P^{|V|}_{|V|}X)) \implies \pi_{-s}(\Sigma^{-V}X)^{tG},\]
where $(-)^{tG} = \widetilde{EG}\smsh_G F(EG_+, - )$ is the \textit{Tate construction}. This is the 
\textit{Tate spectral sequence} of $X$ and we will denote it by 
$\tatess(X)$.  We note that the maps \[EG_+\to S^0\to \widetilde{EG}\] 
give rise to maps of spectral sequences \[\slicess(X)\to\hfpss(X)\to\tatess(X).\]

\subsection*{Acknowledgements}
I would like to thank Jeremy Hahn, Mike Hill, Guchuan Li, XiaoLin Danny Shi, 
and Vesna Stojanoska for some helpful conversations related to this work. 
I would also like to thank Tobias Barthel, and Itamar Mor for some helpful email exchanges related 
to the content of the appendix. 
Finally, I would like express my sincere gratitude to my advisor, Agn\`es Beaudry, 
for recommending this project, for providing invaluable insight throughout, 
and for detailed proof-reading of early drafts. 
This material is based upon work supported by the 
National Science Foundation under Grant No. DMS 2143811.

% --------------------------------------------------------------

%% file: Section2-EqDual.tex
% !TEX root = master.tex

\section{An equivariant dual of Lubin-Tate spectra}\label{duality}

In this section we review some of the theory layed out in 
\cite{dualizingspheres}. The main result identifies the 
equivariant homotopy type of the  $\mK(h)$-local 
Spanier-Whitehead dual of certain Lubin-Tate spectra with respect to 
the action of its stabilizer group.

\subsection{Some conditions on Lubin-Tate spectra}\label{LTcond}
Before delving into the duality, we make some restrictions 
on the Lubin-Tate spectra we consider.
In this paper, we take $E$ to be a Lubin-Tate spectrum associated to 
a pair, $(k,\Gamma)$, where $\Gamma$ is a height $h$ formal group law 
defined over $\F_p$, and $k$ an algebraic extension of $\F_p$. We view $\Gamma$ implicitly as 
a formal group law over $k$ via extension of scalars.
Let $\mathbb{G} = \textrm{Aut}(k,\Gamma)$
denote the corresponding stabilizer group, i.e. the group of automorphisms
of the pair $(k,\Gamma)$.  
We refer the reader to the appendices of \cite{bujard} or 
Appendix 2 of \cite{orangebook} for a 
review of formal group laws and their endomorphism algebras. We will use the results in 
these sources without further mention or proof.

\begin{defn}\label{LTass}
\begin{enumerate}[(a)] 
\item We say that the Lubin-Tate spectrum $E$ \emph{has all automorphisms} if the pair $(k,\Gamma)$ 
is such that the formal group law 
$\Gamma$ has all of its automorphisms over $k$, that is 
\[\textrm{Aut}_{k}(\Gamma)\cong
\textrm{Aut}_{\overline{\F}_p}(\Gamma).\]

\item We say that the Lubin-Tate spectrum $E$ is \emph{residually finite} if $k$ is a finite extension
of $\F_p$. 

\item We say $E$ is a \emph{good} Lubin-Tate spectrum if it is both residually finite and has all 
automorphisms.
\end{enumerate}
\end{defn}
 
If $E$ has all automorphisms then, since $\Gamma$ is defined 
over $\F_p$, we can deduce that the automorphism group $\estab$ splits as
\begin{equation}
\estab\cong\stab\rtimes\Gal(k/\F_p)
\end{equation}
where $\stab =\textrm{Aut}_k(\Gamma)$ is called the \textit{small Morava stabilizer group}. 
The notation here is inherited from the appearance of this group in number theory. 
Recall that $\End_{\overline{\F}_p}(\Gamma)$ is isomorphic to the 
maximal order, $\cO$, in the central divisional algebra 
of Hasse invariant $1/h$ over $\mathbb{Q}_p$. We choose such an isomorphism. 
Since $\Gamma$ is defined over $\F_p$, 
the endomorphism $x^p$ defines an element $\xi_{\Gamma}$ in this 
maximal order called the Frobenius. Furthermore, if $k = \F_{p^r}$
\[\End_{k}(\Gamma)\cong C_{\cO}(\xi_\Gamma^r)\] is the centralizer of $\xi_\Gamma^r$ in 
$\cO$. The following proposition is an immediate consequence of this discussion and is 
useful for determining when a formal group law has all of its automorphisms over a 
finite extension of $\F_p$. 

\begin{prop}\label{frobprop}
The formal group law $\Gamma$ has all automorphisms over $\F_{p^r}$ 
if and only if $\xi_\Gamma^r$ is central in $\cO$, that is, if 
\[\xi_\Gamma^r = p^{r/h}u \in \End_{\F_p}(\Gamma),\] 
for some $u\in\Z_p^\times$. 
\end{prop}

\begin{rmk}
We note that since the Frobenius $\xi_\Gamma$ has valuation $1/h$ 
and the valuation group of $\Z_p$ is $\Z$, if the pair $(\F_{p^r},\Gamma)$ has all automorphisms,
then $r/h\in\Z$.  
\end{rmk}

Our duality results will only directly apply to those Lubin-Tate spectra that 
both have all automorphisms and are residually finite. 
The reason the first of these conditions is useful is that it allows us to 
identify co-operations of $E$ with continuous maps on $\estab$.  
These ideas and results go back to work of Devinatz-Hopkins in \cite{DevinatzHopkins99}. 
We refer the reader to \cite{HoveyEops} for a detailed exposition. 
In particular, the proof of \cite[Theorem~4.11]{HoveyEops} carries over directly. 

\begin{prop}[{\cite[Theorem~2]{DevinatzHopkins99}}]
Let $E$ be a Lubin-Tate spectrum with all automorphisms and let $\estab$ be its
stabilizer group. Then for $H\subset\mathbb{G}$ a closed subgroup we have 
\[\pi_\ast L_{\mK(h)}(E\smsh E^{hH})\cong\map_{c}(\estab/H,E_\ast).\]
\end{prop}

\begin{rmk}
We note that for $A\cong\lim_i A_i$ and $B\cong\lim_i B_i$ topological spaces given 
by the limit of discrete spaces we can identify the set of continuous maps from $A$ to $B$ as
\[\map_{c}(A,B)\cong \lim_i\colim_j\map(A_j,B_i).\]
In the proposition above, $\estab/H$ has its natural profinite topology and 
$E_\ast$ has the $\mathfrak{m}$-adic topology. 
\end{rmk}

\begin{cor}[{\cite[Theorem~3]{DevinatzHopkins99}}]\label{galoiscor}
The unit map
\[L_{\mK(h)}\sphere\to E\] is a $\mK(h)$-local pro-$\mathbb{G}$-Galois extension in the sense of 
\cite[Definition~8.1.1]{rognesgalois}.
\end{cor}

The dependance of our results on $E$ being residually finite is more technical. 
Essentially, without this assumption the proof of 
\cite[Lemma~11.9]{dualizingspheres} may not hold. Because of this, the residually 
finite condition appears to be crucial to these arguments. Indeed, in 
\cref{nonequivariatnappendix} we show that the results of \cref{dualityresults} cannot hold as stated for 
a particular example of $E$ not residually finite. On the other hand, the condition 
that $E$ has all automorphisms merely serves to place us in a familiar setting by 
working with the group $\stab$ rather than an obscure subgroup. 
If $E$ does not have all automorphisms then some version of the results of \cref{dualityresults} should 
still hold with modifications made to work with $\Aut_{k}(\Gamma)$ in place of $\stab$.

\subsection{Dualizing a pro-Galois extension}
In this section we record a general formula for the dual of a $\mK(h)$-local pro-Galois extension 
of spectra. This is essentially the observation that \cite[Lemma~11.4]{dualizingspheres} and its 
proof cary over to this general setting. 
We note that the proof also carries over directly to the case of global (i.e. unlocalized)
pro-Galois extensions though not necessarily to $F$-local extensions for any spectrum $F$. 
We begin by introducing some important maps available to us in the category of 
$\mK(h)$-local spectra with $G$-action, $\spectra^{BG}_{\mK(h)}$, for $G$ a finite group. 

Let $X\in\spectra^{BG}_{\mK(h)}$ be a $\mK(h)$-local spectrum with $G$-action and view 
the fixed point spectra $X^{hH} = F(G/H_+,X^h)^G$ for $H\subseteq G$ as elements of 
$\spectra^{BG}_{\mK(h)}$ with residual $G$-action.

\begin{defn}
Let $K\subseteq H\subseteq G$. 
\begin{enumerate}[(a)]
\item The \textit{restriction} map \[\res:X^{hH}\to X^{hK}\] is defined as 
$F(p_+,X^h)^G$ where $p:G/K\to G/H$ is the quotient map.

\item The \textit{transfer} map \[\trans:X^{hK}\to X^{hH}\] is defined as 
$F(D(p)_+,X^h)^G$ where $D(p):G/H\to G/K$ is the Spanier-Whitehead dual to the quotient map in 
$\spectra^{G}$.
\end{enumerate}
\end{defn}

Now if $R$ is an $\E_\infty$-ring in $\spectra^{BG}_{\mK(h)}$ then the multiplication gives rise 
to a pairing 
\[R^{hH}\smsh_{R^{hG}} R^{hH}\to R^{hH}\xrightarrow{\trans}R^{hG}\] for any $H\leq G$.
Let \[d_H^G:R^{hH}\to F_{R^{hG}}(R^{hH},R^{hG})\] denote the adjoint to the pairing.

\begin{prop}\label{frobeniusdiagram}
Suppose $R$ is an $\E_\infty$-ring in $\spectra^{BG}_{\mK(h)}$ for $G$ a finite group. Let 
$K\subseteq H\subseteq G$. Then there is a commutative diagram
\begin{center}
\begin{tikzcd}
R^{hK}\ar[r,"d_K^G"]\ar[d,"\trans",swap] & F_{R^{hG}}(R^{hK},R^{hG})\ar[d,"F(\res{,} R^{hG})"]\\
R^{hH}\ar[r,"d_H^G"] & F_{R^{hG}}(R^{hH},R^{hG})
\end{tikzcd}
\end{center}
\noindent in the category of $\mK(h)$-local $R^{hG}$-modules with $G$-action, $\Mod_{\mK(h),R^{hG}}^{BG}$. 
\end{prop}

Let $\cG\cong\lim_j\cG_j$ be a profinite group with each $\cG_j$ finite and let 
$A\to B$ be a $\mK(h)$-local pro-$\cG$-Galois extension of commutative ring spectra.  
Then we have a $\cG_j$-Galois extension $A\to B_j$ for each $j$ and 
\[B \simeq L_{\mK(h)}\colim_j B_j\] in the category $\Mod_{\mK(h),A}^{BG}$.

\begin{thm}\label{dualtransferformula}
There is an equivalence in $\Mod_{\mK(h),A}^{BG}$ between the $A$-linear dual of $B$ 
and the following limit over transfer maps 
\[D_A B\simeq \lim_{\trans,j} B_j.\] 
\end{thm}

\begin{proof}
Each $A\to B_j$ is a finite Galois extension with Galois group $\cG_j$
and so \cite[Proposition~6.4.7]{rognesgalois} implies that 
\[D_A B_j\simeq B_j\] with an explicit equivalence given by 
\[d_{e}^{\cG_j}:B_j\to F(B_j,(B_j)^{h\cG_j}) \simeq D_A B.\]
As above, we have that this equivalence is in fact equivariant with respect to the 
natural $\cG$-action. Then since $B = L_{\mK(h)}\colim_j B_j$, \cref{frobeniusdiagram} implies 
that \[\lim_{\trans,j}B_j\simeq \lim_{F(\res,A),j} F_A(B_j,A) 
\simeq D_A B.\] Since this is a limit of $\cG$-equivariant equivalences
the resulting equivalence is itself $\cG$-equivariant. 
\end{proof}

\subsection{Dualizing Lubin-Tate spectra}\label{dualityresults}

Let $E$ be a good Lubin-Tate spectrum.
The goal of this section is to review the arguments of \cite[Section~11]{dualizingspheres}
establishing an equivariant formula for the dual of $E$ with the action by its stabilizer 
group $\estab$. We begin by defining what will turn out to be the dualizing spectrum. 

Let \[\Gamma_i = 1 + p^i\cO.\]
These give a nested sequence of open normal subgroups of $\estab$
\[\cdots\leq\Gamma_{i+1}\leq \Gamma_i\leq \cdots\leq\Gamma_1\leq\Gamma_0 = \stab \leq \estab,\]
which necessarily have finite index. For $j>1$ we have an inclusion of finite groups 
$\Gamma_{i+1}/\Gamma_{i+j}\to\Gamma_i/\Gamma_{i+j}$ 
giving rise to an $\estab$-equivariant transfer
\[\Sigma^\infty_+ B(\Gamma_i/\Gamma_{i+j})\to\Sigma^\infty_+ B(\Gamma_{i+1}/\Gamma_{i+j}).\]
Taking the homotopy limit gives a $\estab$-equivariant transfer
\[\textrm{tr}:\Sigma_+^\infty B\Gamma_i\to \Sigma_+^\infty B\Gamma_{i+1}.\]

\begin{defn}[{\cite[Definition~5.4]{dualizingspheres}}]\label{dualsphere}
Let $I_{\estab}$ be the $\estab$-spectrum given by the $p$-completed colimit in over transfers 
\[I_{\estab} = (\colim_{\text{tr},i}\Sigma^\infty_+ B\Gamma_i)_p^\wedge.\]
\end{defn}

The spectrum of \cref{dualsphere} was tailor-made for the following theorem which generalizes 
the classical Serre duality for continuous group cohomology of compact $p$-adic Lie groups. 

\begin{thm}[{\cite[Theorem~11.16]{dualizingspheres}}]\label{reldual}
There is a $\estab$-equivariant $\mK(h)$-local equivalence 
\[DE \simeq_{\estab} F(I_\estab,E)\simeq_\estab I_{\estab}^{-1}\smsh E,\]
where $\estab$ acts on $DE$ by its action on $E$, by conjugation on $F(I_\estab,E)$, and 
diagonally on $I_{\estab}^{-1}\smsh E$. 
\end{thm} 

\begin{proof}
We provide the reader with a sketch of the proof appearing in \cite[Section~11]{dualizingspheres} 
for convenience. From \cref{dualtransferformula} we have 
\[DE \simeq_{\estab} \textrm{lim}_{\textrm{tr}, i} E^{h\Gamma_i}.\]
The next steps involve decomposing $E$ and rearranging some limits and colimits. 
We start by rewriting $E$ as a $\mK(h)$-local colimit. For this, it is convenient to introduce 
a tower $\{M_I\}$ of generalized Moore spectra as in \cite[Section~4]{hoveystrickland}.
Then we have
\[DE \simeq_\estab \lim_{\trans,i}E^{h\Gamma_i} 
\simeq_\estab \lim_{\trans,i}(L_{\mK(h)}\colim_{\res,j}E^{h\Gamma_j})^{h\Gamma_i}
\simeq_\estab \lim_{\trans,i} \lim_I(\colim_{\res,j}E^{h\Gamma_j}\smsh M_I)^{h\Gamma_i}.\]
We can now use the $\Gamma_i$-HFPSS to show that 
\[(\colim_{\res,j}E^{h\Gamma_j}\smsh M_I)^{h\Gamma_i}\simeq \colim_{\res,j}(E^{h\Gamma_j})^{h\Gamma_i}
\smsh M_I.\]
The key point is that $\pi_\ast E^{h\Gamma_j}\smsh M_I$ are discrete $\Gamma_j$-modules 
and continuous cohomology commutes with filtered colimits of discrete modules. This 
is essentially the content of \cite[Lemma~11.7]{dualizingspheres}. 

We now have 
\[DE\simeq_\estab\lim_I\lim_{\trans,i}\colim_{\res,j}(E^{h\Gamma_j})^{h\Gamma_i}\smsh M_I,\]
and the next step is to swap the limit along transfers with the colimit along restrictions. 
This is the most technically difficult step and we simply reference the proof of 
\cite[Lemma~11.9]{dualizingspheres}. We emphasize that 
this is where the assumption that $E$ is residually finite comes in. The key point is that if 
$E$ is residually finite then both $\pi_\ast(E^{h\Gamma_j}\smsh M_I)$ and $\pi_\ast(E\smsh M_I)$ 
are finite in each degree and so the limit of the $\Gamma_i$-HFPSS along transfers gives 
a spectral sequence as any $\lim^1$-terms will vanish. The proof is then reduced to establishing an 
isomorphism of $E_2$-pages between two spectral sequences: the first being the 
colimit of the limit of the $\Gamma_i$-HFPSS and the second being the limit of the colimit  
of the $\Gamma_i$-HFPSS. 

The last step is to observe that for large enough $i > j$, the $\Gamma_i$ action on 
$E^{h\Gamma_j}$ is trivial and so as one would expect, we have 
\[(E^{h\Gamma_j})^{h\Gamma_i}\simeq_{\estab} F(\Sigma_+^\infty B\Gamma_i,E^{h\Gamma_j}).\]
This is proved in \cite[Lemma~11.12]{dualizingspheres}.

Putting all this together, we have 
\begin{equation}
\begin{split}
    DE &\simeq_\estab L_{\mK(h)}\colim_{\res,j}\lim_{\trans,i}F(\Sigma_+^\infty B\Gamma_i,E^{h\Gamma_j})\\
    &\simeq_\estab L_{\mK(h)}\colim_{\res,j}F(I_\estab,E^{h\Gamma_j})\\
    &\simeq_\estab F(I_\estab,E),
\end{split}
\end{equation}
\noindent where the last step follows from dualizability of $I_\estab$. 
\end{proof}

The applications of this theorem on its own are limited by the fact that the $\estab$-action 
on $I_\estab$ is difficult to get a handle on. 
Using the groups $p^i\cO$ in place of $\Gamma_i$ we can similarly form the following 
spectrum which does not have this disadvantage.

\begin{defn}[\cite{dualizingspheres} Equation 5.7]\label{linearsphere}
Let $S^\lieg$ be the $\estab$-spectrum given by the $p$-complete colimit over transfers
\[S^{\lieg} = (\colim_{\trans,i}\Sigma^\infty_+ B(p^i\cO))_p^\wedge.\]
\end{defn}

The following theorem of Clausen gives us access to the applications 
of the theorems above. See \cite{ClausenAduality} for a proof outline. 

\begin{thm}[Clausen]\label{LinearizationHypothesis}
There is an equivalence of $p$-complete $\estab$-spectra
\[I_\estab\simeq_\estab S^\lieg.\]
\end{thm}

\begin{rmk}
It is clear that this is an underlying equivalence as each of these can be shown to have 
the homotopy type of $p$-adic spheres. 
The authors of \cite{dualizingspheres} proved that this equivalence could be made equivariant 
for certain finite subgroups of $\estab$ using different methods than Clausen. 
See Corollary 11.18 in \cite{dualizingspheres} for a precise statement.  
However, some of our results will require the full power of Clausen's \cref{LinearizationHypothesis}. 
\end{rmk}

\begin{cor}\label{dualofEclausen}
There is a $\estab$-equivariant $\mK(h)$-local equivalence 
\[DE\simeq_\estab F(S^\lieg,E)\simeq_\estab S^{-\lieg}\smsh E.\]
\end{cor}

With this description of the dual at hand the next task is to give an explicit description 
of the action of $\estab$ on $S^\lieg$ for finite subgroups.

\begin{prop}[{\cite[Proposition~9.14]{dualizingspheres}}]\label{RepProp}
Let $G<\stab$ be a finite subgroup. Suppose there is a finitely generated free
abelian group $L\subset\cO$ with the properties that
\begin{enumerate}[(a)]
\item $L$ is stable under the conjugation action of $G$ on $\cO$
\item $\mathbb{Q}_p\otimes L\cong\mathbb{Q}_p\otimes_{\mathbb{Z}_p}\cO$. 
\end{enumerate}
Let $V = \R\otimes L$ and $S^V$ the one-point compactification of $V$. Then there is an 
$G$-equivariant map $S^V\to S^\lieg$ which becomes an equivalence after $p$-completion.
\end{prop}

\begin{defn}
    We will refer to a $G$-representation $V$ as in \cref{RepProp} as a 
    \textit{dualizing $G$-representation} for $E$.
\end{defn}

\begin{thm}
    Let $E$ be a good Lubin-Tate spectrum.
    Let $G\leq\stab$ be a finite subgroup and $V$ be a dualizing $G$-representation for $E$. 
    Then there is a $G$-equivalence \[DE\simeq_G \Sigma^{-V} E.\]
\end{thm}
\begin{proof}
This follows from \cref{dualofEclausen} and \cref{RepProp}.
\end{proof}

\begin{rmk}
Greenlees-Sadofsky Tate vanishing \cite{TateVanishing} allows us 
to apply the contents of this section to study the duals of higher real $K$-theories. 
Explicitly, their vanishing result implies that the canonical map 
\[E_{hG}\xrightarrow{\sim}E^{hG}\]  is an equivalence for any finite group $G$ and any 
$\mK(h)$-local $E$. Thus 
\[D(E^{hG})\simeq D(E_{hG})= F(EG_+\smsh_G E,L_{\mK(h)}\sphere)\simeq F(EG_+,DE)^G = (DE)^{hG}.\]
We will from now on write simply $DE^{hG}$ as there is no ambiguity. 
\end{rmk}

\begin{cor}
Let $E$ be a good Lubin-Tate spectrum.
Let $G<\stab$ be a finite subgroup and $V$ be a dualizing $G$-representation for $E$. 
Then \[DE^{hG}\simeq (\Sigma^{-V} E)^{hG}.\]
\end{cor}

%% file: Section3-DualRepGen.tex
% !TEX root = master.tex

\subsection{The dualizing $C_{2^n}$-representation}\label{therep}

Throughout this section we will work at the prime $p=2$
and set $h=2^{n-1}m$ for some $n\geq 1$ and odd $m\in\N$. In this case, the central
division algebra $\mathbb{D}_h$ of Hasse invariant $1/h$ contains an embedded copy of 
$\mathbb{Q}_2(\zeta)$, where $\zeta =\zeta_{2^n}$ is a primitive $2^n$-th root 
of unity \cite[Theorem~C.6]{bujard}. We choose such an embedding and leave this choice 
implicit throughout this section. The element $\zeta$ then generates 
a subgroup $C_{2^n}\leq\mathcal{O}^\times_h$ \cite[Proposition~1.1]{bujard}.
We refer the reader to \cite{HewettFiniteSubs} and \cite{bujard} for a detailed 
classification of the finite subgoups of the Morava stabilizer groups with warning that some 
statements in both of these references contain typos: specifically, 
\cite[Theorem~5.3]{HewettFiniteSubs} and \cite[Theorem~1.35]{bujard}. 
For a table that accurately summarizes 
the finite subgroups of $\stab$ see \cite[Table~3.14]{BarthelBeaudry}.

In this section, we identify a dualizing $C_{2^n}$-representation, $V_h$, for Lubin-Tate spectra 
at height $h=2^{n-1}m$ and prime $p=2$, then give an explicit description of $V_h$ 
in terms of familiar representations. The case $n=1$ is exceptional since $C_2\leq\stab$ is 
central, so any dualizing $C_2$-representation is always a trivial representation of dimension $h^2$. 
Thus, for the rest of this section, we assume $n\geq 2$. 

The following proposition from the theory of central simple algebras gives a nice description for 
$\mathbb{D}_h$ which is central to our analysis. 
There are many references, but see \cite[Section~IV.3]{milneCFT} for a particularly nice exposition.

\begin{prop}
\label{normalizerprop}
Let $L\subset\mathbb{D}_h$ be a maximal commutative subalgebra of the central
division algebra over $\mathbb{Q}_p$ of Hasse invariant $1/h$. 
Suppose $L/\mathbb{Q}_p$ is a Galois extension and $\Gal(L/\mathbb{Q}_p)$ the Galois group.
Then there exists \[E=\{e_\phi\}_{\phi\in\Gal(L/\mathbb{Q}_p)}\subset\mathbb{D}_h,\] 
such that $E$ is a basis for $\mathbb{D}_h$ as a left $L$-vector space 
and the multiplication on $\mathbb{D}_h$ is determined by 
\begin{enumerate}[(a)]
\item $e_\phi x = \phi(x) e_\phi,\forall x\in L$, and
\item $e_{\phi_1} e_{\phi_2} = \mu(\phi_1,\phi_2)e_{\phi_1\phi_2}$, for some $2$-cocycle $\mu:\Gal(L/\mathbb{Q}_p)^2\to L^\times$.
\end{enumerate}
\end{prop}
\begin{proof}
By the Noether-Skolem theorem, for each $\phi\in\Gal(L/\mathbb{Q}_p)$ 
we may choose an element $e_\phi\in\mathbb{D}_h$
for which \[e_\phi x e_\phi^{-1} = \phi(x), \forall x\in L. \] Note that 
we have 
\[e_{\phi_1}e_{\phi_2} x (e_{\phi_1}e_{\phi_2})^{-1} = \phi_2 (\phi_1(x)),\] so 
the product $e_{\phi_1}e_{\phi_2}$ and our chosen element $e_{\phi_1\phi_2}$ must differ 
by an element $\mu(\phi_1,\phi_2)\in L^\times$, that is
\[e_{\phi_1}e_{\phi_2} = \mu(\phi_1,\phi_2)e_{\phi_1\phi_2}.\]
The cocycle condition is a routine check which we omit. 

To see that $E_{\Gal(L)}$ is an $L$-basis it suffices to check that the $e_\phi$ are linearly 
independent. Suppose 
\[e_\phi = \sum x_\chi e_\chi\] is a finite sum. Then for $x\in L$, 
\[e_\phi x = \phi(x) e_\phi = \phi(x)\sum x_\chi e_\chi = \sum \phi(x) x_\chi e_\chi, \]
and 
\[e_\phi x = \sum x_\chi e_\chi x = \sum x_\chi \chi(x) e_\chi.\]
So we must have $\phi(x) x_\chi = x_\chi \chi(x) $ for all $\chi$ in the sum. Now for $x_\chi\neq 0$ 
this implies $\phi(x) = \chi(x)$ for all $x\in L$ so that $\phi = \chi$. 
\end{proof}

\begin{rmk}
We note that we can choose $E$ so that $e_1 = 1$. We assume this choice in what follows. 
\end{rmk}

Consider the extension $\mathbb{Q}_2(\omega)/\mathbb{Q}_2$, where $\omega$ is a primitive root of 
unity of order $2^m-1$. This is a degree $m$ extension with Galois group $C_m$. 
Since $\mathrm{gcd}(m,2) = 1$, the composite $L = \mathbb{Q}_2(\zeta,\omega)$ must be of degree $h = 2^{n-1}m$ 
over $\mathbb{Q}_2$ and so it is a maximal 
commutative extension of $\mathbb{Q}_2$ in $\mathbb{D}_h$. The extension 
$L/\mathbb{Q}_2$ is then Galois with Galois group 
\[\Gal_h := \Gal(L/\mathbb{Q}_2) \cong C_2\times C_{2^{n-2}}\times C_m,\quad n\geq 2.\]

Let $\tau$ and $\psi$ denote the automorphisms given by 
\[\tau:\begin{cases} \zeta\mapsto \zeta^{-1}\\ \omega\mapsto\omega\end{cases}\quad\text{and}\qquad
\psi:\begin{cases} \zeta\mapsto \zeta^5 \\ \omega\mapsto\omega\end{cases}.\]
We have that $\tau$, generates the canonical subgroup $C_2\leq\Gal_h$ and for 
$n\geq 3$, $\psi$ are generates the canonical $C_{2^{n-2}}\leq\Gal_h$. 
Choose a basis $\{\omega_a\}_{a = 1}^m$ for $\mathbb{Q}_2(\omega)/\mathbb{Q}_2$ and 
let $\nu$ be a generator for $C_m\leq\Gal_h$. 
From \cref{normalizerprop} we may choose a nice $L$-basis of $\mathbb{D}_h$
\[E_h = \{e_\phi|\phi\in\Gal_h\}\subset\mathbb{D}_h,\] then consider the lattice 
\[\mathcal{E}_h = \bigoplus_{i=1}^m \Z\{\zeta\cdot \omega_a e_\phi|\phi\in\Gal_h\},\] where 
\[\zeta\cdot \omega_a e_\phi = \{\omega_a e_\phi,\zeta \omega_a e_\phi,\zeta^2 \omega_a e_\phi,\ldots,\zeta^{h-1} \omega_a e_\phi\}.\]
Using (a) from \cref{normalizerprop} we get the following equation describing the conjugation action by 
$\zeta$ on the generators of $\mathcal{E}_h$.
\begin{equation}
\label{C2naction}
\zeta(\zeta^j \omega_a e_{\tau^i\psi^\ell \nu^b})\zeta^{-1} = \zeta^{1 + j + (-1)^{i+1}5^\ell}\omega_a e_{\tau^i\psi^\ell \nu^b}
\end{equation} 
\noindent for 
\[i\in\{0,1\},\quad \ell\in\{0,1,\ldots,2^{n-2}-1\}, a,b\in\{1,...,m\},\quad\text{and}\quad j\in\{0,1,\ldots,2^{n-1}-1\}. \]

\begin{defn}
We set $V_h = \R\otimes\mathcal{E}_h$ and view this as a $C_{2^n}$-representation with 
action given by conjugation by $\zeta$. We call this the 
\textit{height $h = 2^{n-1}$ dualizing $C_{2^n}$-representation} and its restriction
to a subgroup $H\subset C_{2^n}$ the \textit{height $h=2^{n-1}$ dualizing $H$-representation}.
\end{defn}

\noindent The following theorem justifies this terminology.

\begin{thm}
Let $S^\mathfrak{g}$ be the spectrum of \cref{linearsphere} associated 
to the height $h=2^{n-1}$ Morava stabilizer group. There is a
$C_{2^n}$-equivariant equivalence 
\[S^{\mathfrak{g}}\simeq_{C_{2^n}}S^{V_h}.\]
\end{thm}
\begin{proof}
By \cref{RepProp} it suffices to find a lattice in $\mathcal{O}$ satisfying 
the conditions (1) and (2) in the proposition statement. 
The lattice $\mathcal{E}_h$ satisfies (1) thanks to \cref{C2naction} and (2) thanks to 
\cref{normalizerprop}. However, it may not be the case that 
$\mathcal{E}_h\subset\mathcal{O}$. To address this, choose $r$ such that the $2$-adic 
valuation of any basis element in $\mathcal{E}_h$ is $\geq -r$. 
Then $2^r\mathcal{E}_h\subset\mathcal{O}_h$. Evidently, 
$\mathbb{Q}_2\otimes 2^r\mathcal{E}_h = \mathbb{Q}_2\otimes\mathcal{E}_h$ so the lattice 
$2^r\mathcal{E}_h$ will satisfy (2) since $\mathcal{E}_h$ does. 
Since $2$ is central, the lattice $2^r\mathcal{E}_h$ will still satisfy (1) and the 
$C_{2^n}$-representations 
$\R\otimes 2^k\mathcal{E}_h$ and $\R\otimes\mathcal{E}_h$ are isomorphic.
\end{proof}

Next we analyze the representation $V_h$ and give an explicit description of 
it in terms of familiar representations. We will make use of the following number theory exercise. 

\begin{lem}\label{ntlem}
Assume $n\geq 3$.
\begin{enumerate}[(a)]
\item We have \[5^{2^{n-3}} \equiv 1 + 2^{n-1}\mmod 2^n.\] \label{ntfact1}
\item For $0\leq\ell\leq 2^{n-2}-1$ and $0\leq m\leq 2^{n-2}$, 
\[m(1+5^\ell)\equiv 2^{n-1}\mmod 2^n \quad\text{if and only if}\quad m = 2^{n-2}.\] \label{ntfact2}
\item Assume $n\geq 4$, and let $v_2(-)$ denote the $2$-adic valuation. 
For $1\leq \ell \leq 2^{n-2}-1$, and $0\leq m\leq 2^{n-v_2(\ell)-3}$, 
\[m(1-5^\ell) \equiv 2^{n-1}\mmod 2^n\quad \text{if and only if}\quad m = 2^{n-v_2(\ell)-3}.\] \label{ntfact3}
\end{enumerate}
\end{lem}
\begin{proof}
Part (a) can be proved by a straightforward use of induction on $n$, but also follows from part (c). 

To prove (b) we fix $n$ and 
use induction on $\ell$. For the inductive step note that 
\[2^{n-2}(1+5^{\ell+1}) = 2^{n-2}(1+5^\ell) + 2^{n-2}(4)(5^\ell)\equiv 2^{n-2}(1+5^\ell)\mmod 2^n.\] 
This shows that the modular equivalence holds for $m = 2^{n-2}$. For the reverse direction, suppose 
it holds for a given $m$. Taking $\ell = 2^{n-3}$ and using part (a) implies
\[m(2 + 2^{n-1})\equiv 2^{n-1}\mmod 2^n.\] Since $2^{n-2}+1$ is invertible mod $2^n$, we must have 
$2m\equiv 2^{n-1}(2^{n-2}+1)^{-1}\mmod 2^n$. For $n\geq 3$, $2^n$ divides $2^{2n-3}$, so 
$2m\equiv 2^{n-1}\mmod 2^n$, completing the proof of (b). 

Lastly, we prove (c), which will follow from showing that $v_2(1-5^\ell) = 2 + v_2(\ell)$. 
This follows from the lifting-the-exponent lemma. 
If $\ell$ is odd, then the lemma says that
\[v_2(1-5^\ell) = v_2(1-5) = 2 = 2 + v_2(\ell).\]
If $\ell$ is even, then 
\[v_2(1-5^\ell) = v_2(1-5) + v_2(1+5) + v_2(\ell) -1 = 2 + v_2(\ell). \qedhere\]

\end{proof}

\begin{thm}
\label{thedualizingrep}
For $h=2^{n-1}m$ with $m$ odd and $n\geq 2$, there is an isomorphism of $C_{2^n}$-representations 
\[\R\otimes\mathcal{E}_h \cong V_h = 2^{n-1}m^2(\rho_{C_{2^n}} - \mathrm{Ind}_{C_2}^{C_{2^n}}\sigma).\]
\end{thm}

\begin{proof}
First note that $\zeta$ commutes with $e_{\nu^b}$ and 
with $\omega_a$ for any $a,b\in\{1,...,m\}$. Thus, it suffices to show that there is an isomorphism
\[\R\otimes\Z\{\zeta\cdot e_\phi| \phi\in C_2\times C_{2^{n-1}}\}\cong 
2^{n-1}(\rho_{C_{2^n}}-\mathrm{Ind}_{C_2}^{C_{2^n}}\sigma)\] of representations.
We treat the $n=2$ case separately, where we have 
\[V_h = \R\{\zeta\cdot 1,\zeta\cdot e_\tau\}.\]
Since $\zeta$ commutes with itself, $\R\{\zeta\cdot 1\}$ gives $h$ copies of the trivial representation.
In this case $\zeta^2 = -1$, so that \[\zeta e_\tau \zeta^{-1} = \zeta^2 e_\tau = -e_\tau,\] implying that
$\R\{\zeta\cdot e_\tau\}$ is isomorphic to $h$ copies of the sign representation. 
One can identify $\mathrm{Ind}_{C_2}^{C_4}\sigma_2$ with the $2$-dimensional representation generated by 
$\frac{\pi}{2}$-rotation, so this completes the proof of this case. 

Now on to the $n\geq 3$ case.
The real regular representation $\rho_{2^n}$ of $C_{2^n}$ can be written as
\[\frac{\R[x]}{x^{2^n}-1}\cong\frac{\R[x]}{x^2-1}\oplus\frac{\R[x]}{x^2+1}\oplus\cdots\oplus\frac{\R[x]}{x^{2^{n-3}}+1}\oplus
\frac{\R[x]}{x^{2^{n-2}}+1}\oplus\frac{\R[x]}{x^{2^{n-1}}+1},\] where the generator of $C_{2^n}$ acts by 
multiplication by $x$. 
One can identify the last summand with $\text{Ind}_{C_2}^{C_{2^n}}\sigma_2$, so
our task is to identify $h$ copies of the remaining summands individually within 
\[V_h = \R\{\zeta\cdot 1, \zeta\cdot e_\psi,,...,\zeta\cdot e_{\psi^{2^{n-2}-1}},
\zeta\cdot e_\tau,\zeta\cdot e_{\tau\psi},...,\zeta\cdot e_{\tau\psi^{2^{n-2}-1}}\}.\]

First, by part (\ref{ntfact1}) of \cref{ntlem} and the fact that $\zeta^{2^{n-1}} = -1$, we have 
\[\zeta e_{\psi^{2^{n-3}}}\zeta^{-1} = \zeta^{1-5^{2^{n-3}}}e_{\psi^{2^{n-3}}}
=\zeta^{-2^{n-1}}e_{\psi^{2^{n-3}}}=-e_{\psi^{2^{n-3}}}.\]
So the subspace $\R\{\zeta\cdot 1,\zeta\cdot e_{\psi^{2^{n-3}}}\}\subset V_h$ 
is isomorphic to the representation $h+h\sigma_{2^n}$ corresponding to $h$ copies of 
the first direct summand, $\R[x]/(x^2-1)$. 

Now, using part (\ref{ntfact2}) of \cref{ntlem} and the formula 
\[\zeta^m e_{\tau\psi^\ell}\zeta^{-m} = \zeta^{m(1+5^\ell)}e_{\tau\psi^\ell},\] we deduce 
that subspace spanned by the orbit of $\zeta^j e_{\tau\psi^\ell}$ for some $j,\ell$ under the action above is isomorphic to 
$\R[x]/(x^{2^{n-2}}+1)$. This shows that the subspace $\R\{\zeta\cdot e_{\tau\psi^\ell}|0\leq\ell\leq 2^{n-2}-1\}\subset V_h$ 
is isomorphic to $2^{n-1}\cdot 2^{n-2}/2^{n-2} = h$ copies of the summand $\R[x]/(x^{2^{n-2}}+1)$. 
Note that this completes the proof for $n = 3$, so for the rest of the proof we assume $n\geq 4$.

It remains to find $h$ copies of $\R[x]/(x^{2^r}+1)$ for each $1\leq r\leq n-3$. 
Using part (\ref{ntfact3}) of \cref{ntlem} and the formula 
\[\zeta^m e_{\psi^\ell}\zeta^{-m} = \zeta^{m(1-5^\ell)}e_{\psi^\ell},\] 
we deduce that for each $1\leq\ell\leq 2^{n-2}-1$, with $\ell\neq 2^{n-3}$, the subspace
\[\R\{\zeta\cdot e_{\psi^\ell}\}\] is isomorphic to $2^{v_2(\ell)+2}$ copies of 
$\R[x]/(x^{2^{n-v_2(\ell)-3}}+1)$. It remains to do a dimension count.
For a given $0\leq v < n-3$, let $L_v$ denote the set of $\ell$ such that $1\leq\ell\leq 2^{n-2}-1$ and 
$v_2(\ell) = v$. We must show that 
\[\sum_{\ell\in L_v} 2^{v_2(\ell)+2} = h = 2^{n-1}. \]
Since $v_2(\ell)$ is constant in the sum above, the sum is equivalent to $2^{v+2}|L_v|$. 
Now, $|L_v|$ counts the number of odd $d$ such that $2^v d < 2^{n-2}$, which is 
$2^{n-v-3}$, so we are done. Evidently, for each $0\leq r < n-3$ there is some 
$\ell\in\{1,...,2^{n-2}-1\}$ for which $v_2(\ell) = r$, so this takes care of the remaining summands, 
thus completing the proof.
\end{proof}

As it will play an important role in the later sections, we record the dualizing 
$C_4$-representation at all heights a power of $2$. 

\begin{cor}
\label{thec4rep}
The height $h=2^{n-1}m$ dualizing $C_4$-representation is given by 
\[V_h(C_4) := \mathrm{Res}^{C_{2^n}}_{C_4}V_h\cong \frac{h^2}{2}(1+\sigma_4).\]
\begin{proof}
Taking the restriction of the $C_{2^n}$-representation found in 
\cref{thedualizingrep}, we get
\begin{equation*}
\begin{split}
    2^{n-1}m^2(2^{n-2}\rho_4 - \mathrm{Res}^{C_{2^n}}_{C_4}\text{Ind}^{C_{2^n}}_{C_2}\sigma_2) &\cong 
    2^{n-1}m^2(2^{n-2}\rho_4 - 2^{k-2}\text{Ind}_{C_2}^{C_4}\sigma_2)\\
    &\cong 2^{n-1}m^2(2^{n-2} + 2^{n-2}\sigma_4). \qedhere
\end{split}
\end{equation*}
\end{proof}
\end{cor}

\begin{rmk}
From this corollary, we can deduce that for $h$ even,
\[DE_h\simeq_{C_4}\Sigma^{-\frac{h^2}{2}(1+\sigma_4)}E_h.\] We note that this identification of 
the dualizing representation for $C_4$ does not necessarily depend on 
\cref{LinearizationHypothesis}. It also follows from Theorem 8.11 of \cite{dualizingspheres}.
\end{rmk}

%% file: rblt.tex
\section{The Real bordism theories}
At this point and for the remainder of this paper we work at the 
prime $p=2$ and let $C_{2^n}$ be the cyclic group of order $2^n$ generated 
by  an element $\gamma$. The goal for this section is to introduce some 
computationally convenient Lubin-Tate spectra. With the exception of \cref{funnyprop},
we do not claim originality of any statements in this section. They can all
be found either in work of HHR or \cite{lubintatemodels}. We emphasize that, though it is not explicitly 
stated therein, \cref{factorization} is a direct consequence of the results in \cite{lubintatemodels}. 

Let $\mur$ be the \emph{Real bordism spectrum}. 
This spectrum was first constructed by Landweber and later extensively 
studied by Araki and Hu-Kriz. In particular, 
Araki showed that one could lift the classical Quillen idempotent to this setting 
and produce a $\emph{Real Brown-Peterson spectrum}$, $\bpr$. 
We consider $\mur$ as a 
genuine commutative ring $C_2$-spectrum as in \cite[Chapter~12]{purplebook}. In this 
setting $\bpr$ is a genuine $\mur$-module $C_2$-spectrum. 
We are interested in the norms to $C_{2^n}$
\[\MUCn := N_{C_2}^{C_{2^n}}\mur\quad\text{and}\quad\BPCn := N_{C_2}^{C_{2^n}}\bpr\]
and their quotients. 
We refer the reader to \cite[Section~9.7B]{purplebook} for a definition 
and discussion of the norm and to \cite[Section~2]{slicequotients} for a discussion 
of quotients of norms of $\mur$. 
Certain quotients have been extensively studied in
\cite{hhr17, hishiwax, slicequotients} and these computations will be the key input for 
determining the self-duality of the higher real $K$-theories $E_h^{hC_{2^n}}$.

\subsection{Real bordism Lubin-Tate spectra}
Beaudry-Hill-Shi-Zeng construct Lubin-Tate spectra equipped with equivariant maps 
from $\BPCn$ \cite{lubintatemodels}. We begin this section by reviewing their construction. 

The left unit map  
\[i^\ast_e\bpr\simeq\textrm{BP}\smsh S^0\to
\textrm{BP}^{\smsh 2^{n-1}}\simeq i^\ast_e\BPCn\] of 
underlying spectra gives rise to a $2$-typical formal group law 
$\mathcal{F}$ over $\pi_\ast^e\BPCn$. 
The $C_{2^n}$-action on homotopy
\[f_\gamma:\pi_\ast^e\BPCn\to\pi_\ast^e\BPCn\] comes with 
a canonical strict isomorphism 
\[\psi_\gamma:\mathcal{F}\to f_\gamma^\ast\mathcal{F} =:\mathcal{F}^\gamma.\]
Since these formal group laws are $2$-typical, we have 
\[\psi_\gamma(x) = x + \sideset{}{^{\mathcal{F}^\gamma}}\sum\limits_{i\geq 1} t_i x^{2^i}\] 
for some $t_i$ with $|t_i|=2(2^i-1)$. 
From the arguments of \cite[Section~5.4]{hhr16} we can deduce that  
\[\pi_\ast^e\BPCn = \Z_{(2)}[C_{2^n}\cdot t_1, C_{2^n}\cdot t_2,\ldots],\]
where \[C_{2^n}\cdot x = \{x,\gamma x,\ldots,\gamma^{2^{n-1}-1}x\}\] and the $C_{2^n}$-action is given by 
\begin{equation*}\tag{3.1}
    \label{action}
    \gamma(\gamma^i t_k) = \begin{cases} \gamma^{i+1} t_k & \textrm{for } i < 2^{n-1}-1 \\ 
-t_k & \textrm{for } i = 2^{n-1}-1. \end{cases}
\end{equation*}
In fact, using the arguments of \cite[Section~5]{hhr16}, one can prove something much stronger. 

\begin{prop}
The $C_2$-spectra $i^\ast_{C_2}\BPCn$ are \textit{strongly even} in the sense of 
\cite[Definition~3.1]{hillmeiertmf}. In particular, the restriction maps 
\[\res^{C_2}_e:\pi_{\ast\rho_2}^{C_2}\BPCn\to \pi_{2\ast}^e\BPCn\] are isomorphisms. 
\end{prop}

\noindent We thus have that 
\[\pi_{\ast\rho_2}^{C_2}\BPCn = \Z_{(2)}[C_{2^n}\cdot\tgen_1,C_{2^n}\cdot\tgen_2,\ldots]\] for 
some $\tgen_i$ with $|\tgen_i| = (2^i-1)\rho_2$. 

Now for $m\geq 1$ and $k/\F_2$ an algebraic extension, we consider the graded ring \[R_n(k,m) = 
W(k)[C_{2^n}\cdot t_1,\ldots,C_{2^n}\cdot t_{m-1},
C_{2^n}\cdot u][C_{2^n}\cdot u^{-1}]_{\mathfrak{m}}^\smsh\]
where \[\mathfrak{m} = (C_{2^n}\cdot t_1,\ldots,C_{2^n}\cdot t_{m-1}, C_{2^n}(u-\gamma u)),\]
and $|t_i| = 2(2^i-1)$ for $1\leq i\leq m-1$ and $|u| = 2$. 
There is a natural action of $C_{2^n}$ on $R_n(k,m)$ given just as in $\pi^e_\ast\BPCn$ on the 
generators and extended $W(k)$-linearly. 
There results a $C_{2^n}$-equivariant map 
\begin{equation}\tag{3.2}
\label{ringmap}
\pi_\ast^e \BPCn\to R_n(k,m)
\end{equation}
given by 
\[t_i\mapsto\begin{cases}
t_i & 1\leq i\leq m-1 \\ u^{2^m-1} & i=m \\ 0 & i > m
\end{cases}.\] 

\noindent This gives rise to formal group laws $F_h$ over $R_n(k,m)$ and $\Gamma_h$ over 
$K_n(k,m) = R_n(k,m)/\mathfrak{m}$, where $h = 2^{n-1}m$ is the height of $\Gamma_h$. 
Proposition 2.12 of \cite{lubintatemodels} shows that they 
both come equipped with an action of $C_{2^n}$ extending the $C_2$-conjugation action.
In particular the $C_{2^n}$-action on $\Gamma_h$ is generated by 
\begin{equation}\tag{3.3}
\label{fglaction} 
x +_{\Gamma_h} u^{2^m-1}x^{2^m}.
\end{equation}

\begin{convention*}
When discussing the height of the formal group law $\Gamma_h$ above, we view 
the height variable $h$ as implicitly depending on some $n$ and $m$
via $h = 2^{n-1} m$. When we wish to emphasize the particular choice of $n$ and $m$, we will 
write $h(n,m)$ instead of $h$. 
\end{convention*}

\begin{thm}[{\cite[Theorem~4.6]{lubintatemodels}}]
The formal group law $F_h$ is a universal deformation of the height $h=h(n,m)$ 
formal group law $\Gamma_h$. Furthermore $C_{2^n}$ can be identified as a subgroup of 
the group of automorphisms of the pair $(k,\Gamma)$:
\[C_{2^n}\leq \mathbb{G}(k,\Gamma_h).\] 
\end{thm}

\begin{rmk}
Throughout \cite{lubintatemodels} the authors work over a finite perfect field 
of characteristic $2$. However, the finiteness condition is unnecessary for this result.
\end{rmk}
 
The Goerss-Hopkins-Miller theorem then gives us a Lubin-Tate
commutative ring spectrum with an action of $C_{2^n}$ given explicitly on homotopy 
groups by \cref{action}.

\begin{defn}\label{rblt}
The height $h = h(n,m)$ \textit{Real bordism Lubin-Tate spectrum}, denoted $E^{\R}_h(k)$, is the 
Lubin-Tate spectrum associated to $(k,\Gamma_h)$. When $k = \overline{\F}_2$, we simply 
write $E^{\R}_h$. 
\end{defn}

\begin{rmk}
We note that the results of \cref{dualityresults} will not apply directly to $E^{\R}_h$ (See \cref{nonequivariatnappendix}).
We will simply use $E^{\R}_h$ as a means of transfering information from $\BPCn$ to 
suitable Lubin-Tate spectra (See \cref{uvsLT}).
Nevertheless, it is natural to ask: for which 
$h=h(n,m)$ does $\Gamma_h$ have all of its automorphisms over a finite field $k$? 
We end this section with an answer to this question. 
\end{rmk}

\begin{prop}\label{funnyprop}
Let $h=h(n,m)$ and $\Gamma_h$ be the formal group law over $K(k,m) = R_n(k,m)/\mathfrak{m}$ 
constructed above with an explicit 
$C_{2^n}$-action. Then if $n = 1$, the pair $(\F_{2^h},\Gamma_h)$ has all automorphisms, and 
if $n=2$, the pair $(\F_{2^{2h}},\Gamma_h)$ has all automorphisms. Furthermore, for all $n>2$, 
there is no finite $r$ such that the pair $(\F_{2^r},\Gamma_h)$ has all automorphisms.  
\end{prop}
\begin{proof}
We use the equivalence of categories constructed in \cite[Section~2]{lubintatemodels}. 
This allows us to go 
from working with graded rings to ungraded rings by considering the formal group law over 
$K(k,m)_0$ given by $\widetilde{\Gamma}_h = u\Gamma u^{-1}$. 
The action of \cref{fglaction} translates to the power series 
\[\psi_\gamma(x) = x +_{\widetilde{\Gamma}_h} x^{2^m}\] over $K(m)_0$. From this description, 
we see that under the identification \[\text{End}_{\overline{\F}_2}(\Gamma_h)\cong\mathcal{O}\] we have
\[\zeta_{2^n} - 1 = \xi^m,\] for some primitive root $\zeta_{2^n}$.    
Clearly if $n=1$, then this element is central, and if $n=2$, then $\xi^{4m} = -4$ is central.  
It remains to show that no power of $\pi = \zeta_{2^n}-1$ lies in $\Z_2^\times$ for $n > 2$ \cite{stackanswer}. 
To see this first note that 
if $\pi^r\in\Z_2^\times$ then for any $\sigma\in\textrm{Gal}(\mathbb{Q}_2(\zeta_{2^n})/\mathbb{Q}_2)$,
$\frac{\sigma(\pi)}{\pi}$ is an $r$th root of unity. Taking $\sigma$ such that 
$\sigma(\zeta_{2^n}) = \zeta_{2^n}^3$, we find that \[\frac{\sigma(\zeta_{2^n}-1)}{\zeta_{2^n}-1} = 
1 + \zeta_{2^n} + \zeta_{2^n}^2\] is an $r$th root of unity. It follows that 
$\zeta_{2^n}^{-1} + 1 + \zeta_{2^n}$ is a root of unity in $\Z_2^\times$ hence must be $\pm 1$. 
If it is $+1$ then we must have $n=2$ and for $-1$ we must have $n=1$.  
\end{proof}

\subsection{The $C_{2^n}$-orientation}
In order to make use of the known slice computations for Real bordism theories and their norms,
we will need to make use of an equivariant spectrum-level version of the ring map
\[\pi_\ast^e\BPCn\to\pi_\ast^e E^{\R}_h\] from \cref{ringmap}. 
We will call such a map a \textit{$C_{2^n}$-orientation} of $E^{\R}_h$, and its construction 
is a consequence of the Hahn-Shi Real orientation \cite{realorientation}.

\begin{prop}[{\cite[Theorem~1.7]{lubintatemodels}}]\label{realorientationmap}
There is a $C_{2^n}$-equivariant map 
\[\phi:\BPCn\to E^{\R}_h\] inducing 
the ring map of \cref{ringmap} on underlying homotopy groups. 
\end{prop}
\begin{proof}
We include a short proof sketch for completeness. The Real-orientation 
theorem \cite{realorientation} gives us a $C_2$-equivariant map 
\[\bpr\to i^\ast_{C_2} E^{\R}_h.\] Using the norm-restriction adjunction 
on equivariant commutative ring spectra we get a $C_{2^n}$-equivariant
map
\[\BPCn\to N_{C_2}^{C_{2^n}}i^\ast_{C_2}E^{\R}_h\to E^{\R}_h.\] 
Since this is $C_{2^n}$-equivariant, the coefficients of the automorphism 
$\psi_\gamma$ in $\pi_\ast^e\BPCn$ must map to the coefficients of the 
corresponding automorphism of $F_h$. This precisely specifies the 
ring map of \cref{ringmap}.
\end{proof}

We will also find it beneficial to establish that some evident factorizations
of the algebraic map in \cref{ringmap} can be lifted to spectrum-level factorizations.
Namely, the $C_{2^n}$-orientation $\BPCn\to E^{\R}_h$ should factor through

\begin{enumerate}
    \item the map to the \emph{localization} of $\BPCn$ inverting $N(\tgen_m)$, and
    \item the \emph{quotient} of $\BPCn$ by the equivariant ideal generated by $\tgen_i$ for $i> m$. 
\end{enumerate}
The factorization through $(1)$ was shown in \cite{lubintatemodels} and we will simply 
quote their result. In fact, they prove something even stronger. 

\begin{notation*}
    For each $1\leq r< n$, using the unit of the norm-restrict adjuction we have a map
    \[\mathrm{BP}^{((C_{2^r}))}\to\mathrm{BP}^{((C_{2^n}))}.\] 
    Let $\tgen^{C_{2^r}}_i$ denote the image of the generators of 
    $\pi_{\ast\rho_2}^{C_2}\mathrm{BP}^{((C_{2^r}))}$ under this map.
\end{notation*}

\begin{prop}[{\cite[Proposition~6.3]{lubintatemodels}}]\label{localizationfactorization}
The $C_{2^n}$-orientation of $E^{\R}_h$ factors through the localization 
\begin{center}
\begin{tikzcd}
\BPCn\ar[r]\ar[d] & E^{\R}_h\\
D_h^{-1}\BPCn\ar[ur]
\end{tikzcd}
\end{center}
\noindent inverting an element $D_h$ divisible by
\[N(\tgen_h^{C_2})N(\tgen_{h/2}^{C_4})\cdots N(\tgen_m).\]
In particular, the $C_{2^n}$-orientation factors through the localization inverting $N(\tgen_m)$. 
\end{prop}

We now move on to $(2)$. The author would like to thank Jeremy Hahn for outlining the argument to us. 
By a \emph{quotient} of a module over a 
commutative $G$-ring spectrum, we mean in the sense of \cite[Section~10.10]{purplebook}, to which we refer 
the interested reader to for an in-depth discussion. 

\begin{defn}
Let $R$ be a $G$-equivariant commutative ring spectrum. For $x\in\pi_\star^{H}R$, let 
$R[G\cdot x]$ denote the \emph{twisted monoid ring} generated by $x$. 
Then for any $R$-module, $M$, the \emph{quotient} of $M$ by $(G\cdot x)$ is given by 
\[M/(G\cdot x) := M\smsh_{R[x]} R. \]
\end{defn}

\begin{defn}
The quotient of $\BPCn$ by $(C_{2^n}\cdot\tgen_{m+1},\ldots)$ is defined to be 
\[\BPCn\langle m\rangle: = 
\colim_w \frac{\BPCn}{(C_{2^n}\cdot\tgen_{m+1})}\smsh_{\MUCn}\cdots\smsh_{\MUCn}\frac{\BPCn}{(C_{2^n}\cdot\tgen_{m+w})}.\]
\end{defn}

\begin{rmk}
In order to prove $(2)$, we make the following assumption. 
See \cite{roytman} for some work in-progress in this direction. 
\end{rmk}

\begin{ass}\label{e2rho}
The real orientation map \[\mathrm{MU}_\R\to E^{\R}_h\] is an $\mathbb{E}_{2}$-ring map. 
\end{ass} 

\begin{prop}\label{quotientfactorization}
    Given \cref{e2rho}, the $C_{2^n}$-orientation factors as
    \begin{center}
    \begin{tikzcd}
    \BPCn\ar[r,"\phi"]\ar[d] & E^{\R}_h \\
    \BPCn\langle m\rangle\ar[ur]
    \end{tikzcd}
    \end{center}
\end{prop}

\begin{proof}
One can verify that on the $C_2$-level, the quotient $i^\ast_2\BPCn/(C_2\cdot\tgen_i)$ is simply the 
cofiber of the multiplication by $\tgen_i$ map. 
Since each $\tgen_i$ maps to zero under $\phi$, this implies a factorization 
\begin{center}
    \begin{tikzcd}
    i^\ast_2\BPCn\ar[r,"i^\ast_2\phi"]\ar[d] & i^\ast_2E^{\R}_h \\
    i^\ast_2\BPCn/(C_2\cdot\tgen_i)\ar[ur]
    \end{tikzcd},
\end{center} 
which can be normed up to a diagram 
\begin{center}
    \begin{tikzcd}
    N_2^{2^n}i^\ast_2 \BPCn\ar[r]\ar[d] & N_2^{2^n}i^\ast_2 E^{\R}_h\\
    \BPCn/(C_{2^n}\cdot\tgen_i)\ar[ur]
    \end{tikzcd}.
\end{center} 
Then, the norm of the unit of the norm-restriction adjuction gives a map 
\[N_2^{2^n}i^\ast_2 \BPCn = N_2^{2^n}i^\ast_2 N_2^{2^n}\bpr\to N_2^{2^n}\bpr = \BPCn.\]
Using this, together with the counit of the norm-restriction adjunction on $E_h^{\R}$, one can 
form the factorization
\begin{center}
    \begin{tikzcd}
    \BPCn\ar[r,"\phi"]\ar[d] & E^{\R}_h \\
    \BPCn/(C_{2^n}\cdot\tgen_i)\ar[ur]
    \end{tikzcd},
\end{center} 
\noindent for each $i > m$. From \cref{e2rho} and the fact that the norm is symmetric monoidal, we 
have that $\MUCn\to E^{\R}_h$ is also an $\mathbb{E}_2$-ring map. Therefore, $E^{\R}_h$ has an 
$\MUCn$-linear multiplication and we can form 
\begin{equation*}
\begin{split}
\frac{\BPCn}{(C_{2^n}\cdot\tgen_{m+1})}\smsh_{\MUCn} &\cdots\smsh_{\MUCn} \frac{\BPCn}{(C_{2^n}\cdot\tgen_{m+w})} \\
&\to E^{\R}_h\smsh_{\MUCn}\cdots\smsh_{\MUCn} E^{\R}_h\to E^{\R}_h.
\end{split}
\end{equation*} Taking the colimit gives the desired 
factorization.
\end{proof}

\noindent Lastly, we combine \cref{localizationfactorization} and \cref{quotientfactorization}. 

\begin{thm}\label{factorization}
The $C_{2^n}$-orientation factors as 
\begin{center}
\begin{tikzcd}
\BPCn\ar[r]\ar[d] & E^{\R}_h \\
D_h^{-1}\BPCn\langle m\rangle\ar[ur]
\end{tikzcd}
\end{center}
\end{thm}

\begin{proof}
Starting with the factorizaton from \cref{quotientfactorization}, the proof of 
\cref{localizationfactorization} given in \cite[Proposition~6.3]{lubintatemodels} 
goes through verbatim with the quotient, $\BPCn\langle m\rangle$, in place 
of $\BPCn$. 
\end{proof}

\begin{notation*}
    For the sake of legibility, we will write $E_{C_{2^n}}(m)$ for 
    $D_h^{-1}\BPCn\langle m\rangle$. The notation is meant to make the reader think 
    of these as $C_{2^n}$-equivariant analogs of Johnson-Wilson spectra. 
\end{notation*}

%% file: Section4-Orientability.tex
% !TEX root = master.tex
\section{Equivariant periodicities and integer shifts}\label{orientability}

\subsection{$\text{RO}(G)$-periodicities and where to find them}

The purpose of this section is to review some of the theory layed out in 
\cite[Section~3]{bbhs}. Throughout, we take $G$ to be a finite group. 
Let $R\in\textrm{Sp}^G$ be a commutative ring. We can then consider
the symmetric monoidal category of modules over $R$ \cite{BlumbergHillModuleCats}, 
its associated group of invertible modules $\textrm{Pic}_G(R)$, and obtain what we will call 
the \textit{$J$-homomorphism}: 
\begin{equation*}
\begin{split}
J:RO(G)&\to\textrm{Pic}_G(R)\\
V &\mapsto \Sigma^V R.
\end{split}
\end{equation*}

\begin{prop}[{\cite[Proposition~3.1]{bbhs}}]
Let $R^{hG}\to R$ be a faithful $G$-Galois extension for a finite group $G$ and 
view $R$ as a (cofree) genuine $G$-spectrum. Then 
$\mathrm{Mod}(R^{hG})$ and $\mathrm{Mod}_G(R)$ are equivalent categories and there 
is an isomorphism
\begin{equation*}
\begin{split}
\mathrm{Pic}_G(R) &\cong\mathrm{Pic}(R^{hG}) \\
M &\mapsto M^{hG}.
\end{split}
\end{equation*}
\end{prop}

\begin{rmk}\label{cofreeJhom}
Thus, under the assumptions of the proposition, we may view the $J$-homomorphism
as a map 
\begin{equation*}
\begin{split}
J:\mathrm{RO}(G) &\to\mathrm{Pic}(R^{hG})\\ 
V &\mapsto (\Sigma^V R)^{hG}.
\end{split}
\end{equation*}
We will not notationally distinguish these two maps.  
\end{rmk}

\begin{defn}\label{periodicdef}
We say $R$ is \textit{$V$-periodic} if $V\in RO(G)$ is in 
the kernel of the $J$-homomorphism, that is, if $\Sigma^V R\simeq_G R$. 
We refer to elements $V$ in the kernel of $J$ as \emph{$\mathrm{RO}(G)$-periodicities} of $R$. 
\end{defn}

\noindent Our goal for this section is to determine a useful condition that will 
allow us to deduce some $\mathrm{RO}(G)$-periodicites. We begin with the following observation.

\begin{prop}\label{equivalenceobservation}
There is a $G$-equivalence of $R$-modules \[u:\Sigma^{|V|}R \xrightarrow{\sim_G} \Sigma^V R\] 
if and only if there is a $G$-equivariant map 
\[\hat{u}:S^{|V|}\to \Sigma^V R\] such that the composition
\[\Sigma^{|V|}R\xrightarrow{\hat{u}\smsh\eta} \Sigma^V R\smsh R\xrightarrow{1\smsh\mu} \Sigma^V R\]
is a $G$-equivalence of $R$-modules. 
\begin{proof}
One direction is evident, namely, if we have such a $\hat{u}$ then the composition 
gives the equivalence $u$. For the other direction
consider the diagram below. 

\begin{center}
\begin{tikzcd}
    \Sigma^{|V|}R\ar[r,"\hat{u}\smsh 1"]\ar[d,"1\smsh\eta\smsh 1",swap] & 
    \Sigma^V R\smsh R\ar[r,"1\smsh\mu"] & \Sigma^V R\\
    \Sigma^{|V|}R\smsh R\ar[ur,"u\smsh 1"]\ar[rr,"1\smsh\mu"] & &\Sigma^{|V|}R\ar[u,"u"]
\end{tikzcd}
\end{center}

If we have an equivalence $u$, then we define 
$\hat{u}$ by precomposition of $u$ with \[1\smsh\eta:S^{|V|}\to \Sigma^{|V|}R.\] 
The map $\hat{u}$ is 
defined so that the triangle on the left commutes. Since $u$ is a map
of $R$-modules the trapezoid on the right commutes. It follows that the 
border rectangle commutes, implying that the top edge of this rectangle is
an equivalence. 
\end{proof}
\end{prop}

\begin{defn}{(\cite[Definition~3.3]{bbhs})}
An \textit{$R$-orientation} of $V$ is an element in $\pi_{|V|}^G\Sigma^V R$ which 
is invertible when viewed as an element of the $RO(G)$-graded ring $\pi_\star^G R$ under the 
isomorphism $\pi_{|V|}^G \Sigma^V R\cong\pi_{|V|-V}^G R$. 
We say $V$ is \textit{$R$-orientable} if there exists an $R$-orientation of $V$. 
\end{defn}

\begin{rmk}
We should note that there are choices involved when working with $\text{RO}(G)$-graded homotopy groups 
such as the isomorphisms $\pi_{|V|}^G\Sigma^V R\cong\pi_{|V|-V}^G R$ used in the definition above. That such choices 
can be made consistently has been proved, for example, in \cite{DuggerSigns}. 
Nevertheless, the property of being an $R$-orientation is invariant under the choice of isomorphisms, 
for any two choices will differ by multiplication by a unit in the Burnside ring $\pi_0^G \sphere$.  
\end{rmk}

\begin{prop}
$R$ is $(|V|-V)$-periodic if and only if $V$ is $R$-orientable.
\end{prop}

\begin{proof}
Suppose $\hat{u}\in\pi_{|V|}^G\Sigma^V R$ is an $R$-orientation of $V$. Then $|V|$-suspending 
the inverse of $\Sigma^{-V}\hat{u}$ gives an element $\hat{u}'\in\pi_V^G\Sigma^{|V|}R$ and the 
composition 
\[\Sigma^V R\xrightarrow{\hat{u}'\smsh\eta}\Sigma^{|V|}R\smsh R\xrightarrow{1\smsh\mu}\Sigma^{|V|}R\] 
is an inverse equivalence to the composition of \cref{equivalenceobservation}.

Now suppose $\Sigma^{|V|}R\to \Sigma^V R$ is a $G$-equivalence of $R$-modules. Then smashing 
this equivalence with the inverse of the source gives a map 
\[R\to \Sigma^{V-|V|}R\] which by precomposition with the unit corresponds to an element in 
$\pi_{|V|}^G\Sigma^V R$. 
\end{proof}

\begin{rmk}\label{cofreeRorientations}
Assuming $R$ is cofree 
and that $R^{hG}\to R$ is a faithful $G$-Galois extension,
then, as argued in the paragraph following Theorem 6.4 in \cite{hillmeiertmf}, $\Phi^G R$ is contractible
so by Corollary 10.6 in \cite{hhr16}, $\Sigma^V R$ is also cofree.
This gives us an isomorphism 
\[\pi_\ast^G(\Sigma^V R)\to \pi_\ast^G F(EG_+,\Sigma^V R)\cong\pi_\ast(\Sigma^V R)^{hG}.\]
We may then view an $R$-orientation as a class in 
$\pi_{|V|}(\Sigma^V R)^{hG}$. 
\end{rmk}

Now consider $V\in \mathrm{RO}(G)$ a classically orientable representation, i.e. one for which
the composition \[G\xrightarrow{V}O(\dim(V))\xrightarrow{\mathrm{det}}\{\pm 1\}\] is trivial. 
Choose an identification $t:i^\ast_e S^{|V|}\to i^\ast_e S^V$ of underlying topological spaces.
For any $R$, this defines an element $[t]\in\pi_{|V|}^e\Sigma^V R$ by smashing $t$ with the unit 
map. Although $t$ itself is not $G$-equivariant, since $V$ is classically orientable $t$ is $G$-equivariant up to homotopy: 
\[g\cdot[t] = [gtg^{-1}] = [t].\] The class
$[t]$ is then in the $G$-fixed points of the resulting $G$-action on the 
homotopy groups
\[ [t]\in(\pi_{|V|}^e\Sigma^V R)^G.\] 
Note that $[t]$ is invertible when viewed as an element of 
$\pi_0^e R$. This leads us to consider the following definition. 

\begin{defn}{(\cite[Definition~3.7]{bbhs})}
Let $V$ be a classically orientable representation. 
A \textit{pseudo-$R$-orientation} of $V$ is an element in 
$(\pi_{|V|}^e\Sigma^V R)^G$ which is invertible when viewed as an element 
of $(\pi_0^e R)^G$ under the isomorphism 
$\pi_{|V|}^e\Sigma^V R\cong\pi_0^e R$. 
We say $V$ is \textit{pseudo-$R$-orientable}
if there exists a pseudo-$R$-orientation of $V$. 
\end{defn}

The discussion above shows that if $V$ is classically orientable then 
a choice of orientation gives rise to a pseudo-$R$-orientation for any $R$.
The question we are interested in is when a pseudo-$R$-orientation can be 
refined to a $G$-equivariant map, that is, to an $R$-orientation.

\begin{prop}[{\cite[Proposition~3.10]{bbhs}}]
\label{orientationprop}
Let $V$ be a classically orientable $G$-representation and $R$ be 
cofree with $R^{hG}\to R$ a faithful $G$-Galois extension. 
Then $V$ is $R$-orientable if and only if there exists a 
pseudo-$R$-orientation
\[\tilde{u}_V\in H^0(G,\pi_{|V|}^e\Sigma^V R)\] 
which is a permanent cycle in the HFPSS
\[H^s(G,\pi_t(\Sigma^V R))\implies \pi_{t-s}(\Sigma^V R)^{hG}.\]
\end{prop}

The Lubin-Tate spectra that we are interested in are cofree and the 
Galois condition is satisfied by \cref{galoiscor}. 
So we may use \cref{orientationprop} directly to study them. 
However, we are still missing a crucial ingredient. 
We wish make use of the techniques of HHR and the known slice spectral
sequence computations for norms of Real bordism and their quotients. 
It is therefore helpful to define certain classes at the level of the 
slice spectral sequence that play the role of pseudo-$R$-orientations.

Let us continue with $V$ a classically orientable $G$-representation. 
In this case, the restriction map gives an isomorphism 
\[res^G_e:\pi_{|V|}^G\Sigma^V H\underline{\Z}\xrightarrow{\cong}
\pi_{|V|}^e\Sigma^V H\underline{\Z}\cong\Z,\] where a choice of orientation
is equivalent to an isomorphism of the target with $\Z$. 

\begin{defn}
\textnormal{For a classically oriented $V$, we define the \textit{orientation class} of $V$, denoted
$u_V\in\pi^G_{|V|}\Sigma^V H\underline{\Z}$, to be the 
preimage of $1$ under the restriction isomorphism above. }
\end{defn}

\noindent It turns out that the orientation classes become invertible after cofree 
localization.

\begin{prop}\label{uvinvert}
The image of $u_V$ under the cofree localization map 
\[H\underline{\Z}\to H\underline{\Z}^h\] is a unit. 
\end{prop}
\begin{proof}
The homotopy groups $\pi^G_{|V|}\Sigma^V H\underline{\Z}^h$ 
are given by the Borel homology groups 
\[H\Z_{|V|}(EG_+\smsh_G S^{-V}).\] We can identify $EG_+\smsh_G S^V$ with 
the Thom spectrum of $EG \times_G -V\to BG$ and the image of $u_V$ is 
the Thom class of this vector bundle. The Thom class of $EG\times_G V\to BG$
is then an inverse to $u_V$. 
\end{proof}

\begin{rmk}
More generally, the argument above shows that if we view an ordinary spectrum $E$ as 
a cofree $G$-spectrum with trivial action, then the notion of $E$-orientability of a 
$G$-representation $V$ is equivalent to the usual notion of orientability for the 
corresponding vector bundle over the classifying space $BG$. 
\end{rmk}

Now let us get back to our commutative $G$-ring $R$. From \cite[Corollary~4.54]{hhr16}
we know that the $0$-slice of the sphere is 
\[P^0_0 S^0\simeq H\underline{\Z}.\] Thus for any classically oriented $V$
the unit map $S^0\to R$ gives rise to classes
\[u_V\in\pi_{|V|-V}^G P^0_0 R\] 
in filtration zero of the $E_2$-page of $\textrm{SliceSS}(R)$.
We will not notationally or terminologically distinguish the $u_V$ classes 
with their images in various rings $R$. 
The cofree localization map gives rise to a morphism of spectral sequences 
\[\textrm{SliceSS}(R)\to\textrm{SliceSS}(R^h) = \mathrm{HFPSS}(R).\]
The following is then an immediate corollary of \cref{uvinvert} and 
\cref{orientationprop}

\begin{prop}\label{uvorientability}
Let $R$ be a cofree $G$-ring spectrum with $R^{hG}\to R$ a faithful 
$G$-Galois extension. Let $V$ be a classically orientable $G$-representation. 
If the class $u_V$ in $\mathrm{HFPSS}(R)$ is a permanent cycle, then $V$ is $R$-orientable. 
\end{prop}

\begin{rmk}\label{eulerclasses}
We note that the inclusion $S^0\to S^V$ of the fixed points for any representation 
$V$ of $G$  gives rise to a class $a_V\in\pi_{-V}S^0$, and thus to a class in filtration $s = |V|$ of any 
of the above spectral sequences for any ring spectrum. These are called \emph{Euler classes}. By definition, 
these must be permanent cycles in the slice spectral sequence of the sphere, 
hence too in the slice spectral sequence for any ring spectrum. 
For more on these classes, see \cite{hhr17}.
\end{rmk}

\subsection{Orientation classes for Lubin-Tate spectra}\label{uvsLT}
The takeaway from the preceding section is that we have orientation classes $u_V$ in filtration zero 
of the slice spectral sequence for $\BPCn$, which map to orientation classes in the 
homotopy fixed point spectral sequence of $E^{\R}_h$ via the
$C_{2^n}$-orientation from \cref{realorientationmap}. 
The orientation classes in the spectral sequence for $\BPCn$ come from pushing 
forward those in that of $\MUCn$ along the quotinent map killing the generators sent to 
zero in the Quillen idempotent. We similarly define orientation classes in the slice spectral 
sequence of any quotient or localization of $\BPCn$. 

\begin{prop}
    \label{oneuv}
    Let $V$ be a classically orientable $C_{2^n}$-representation and let 
    $E_h$ be a good Lubin-Tate theory at height $h = 2^{n-1}m$. 
    If $u_V$ is a permanent cycle in 
    $\mathrm{HFPSS}(\ECnm)$ then $E_h$ is $(|V|-V)$-periodic: 
    \[\Sigma^{|V|-V}E_h\simeq_{C_{2^n}}E_h\]
\end{prop}
\begin{proof}
\cref{factorization} provides a morphism of spectral sequences 
\[C_{2^n}\textrm{-}\mathrm{HFPSS}(\ECnm)\to C_{2^n}\textrm{-}\mathrm{HFPSS}(E^{\R}_h).\]
Since permanent cycles map to permanent cycles, $u_V$ is a permanent cycle in the target spectral 
sequence.

Let $(k,\Gamma)$ be the pair associated to $E_h$. 
After base change to $\overline{\F}_2$, $\Gamma$ will be isomorphic 
to $\Gamma_h$, and so we may choose an equivalence
\[E^{\R}_h\xrightarrow{\sim} E_{(\overline{\F}_2,\Gamma)}\] 
which is $C_{2^n}$-equivariant with respect to an appropriate choice of embedding of $C_{2^n}$ in 
$\Aut_{\overline{\F}_2}(\Gamma)\leq \estab(\overline{\F}_2,\Gamma)$. We then have an isomorphism of spectral sequences 
\begin{equation*}
%\label{ssmapLTs}
    C_{2^n}\textrm{-}\mathrm{HFPSS}(E^{\R}_h)\xrightarrow{\cong}C_{2^n}\textrm{-}\mathrm{HFPSS}(E_{(\overline{\F}_2,\Gamma)}).
\end{equation*}
Thus the orienation class $u_V$ is a permanent cycle in the target spectral sequence.

Now, we also have a morphism $E_h\to E_{(\overline{\F}_2,\Gamma)}$ inducing a morphism of spectral sequences 
\begin{equation*}\tag{4.1}
    \label{ssmapLTs}
        C_{2^n}\textrm{-}\mathrm{HFPSS}(E_h)\xrightarrow{\cong}C_{2^n}\textrm{-}\mathrm{HFPSS}(E_{(\overline{\F}_2,\Gamma)}),
\end{equation*}
preserving orientation classes. Thus, if $u_V$ supports a differential in the source of \cref{ssmapLTs} then 
the tartget of that differential must map to zero.
By Galois descent, the $E_2$-page of the target of \cref{ssmapLTs} given by 
\begin{equation*}\tag{4.2}
\label{ssGalcl}
H^\ast(C_{2^n};\pi_\star E_{(\overline{\F}_2,\Gamma)})\cong H^\ast(C_{2^n};\pi_\star E_{(\overline{\F}_2,\Gamma)})^{\Gal}\ox_{W(k)}W(\overline{\F}_2),
\end{equation*}
where $\Gal = \Gal(\overline{\F}_2/k)$. Since $\Gamma$ has all of its automorphisms over $k$ by assumption, 
we can extend the $C_{2^n}$ to a subgroup $G = C_{2^n}\times\Gal\leq\estab(\overline{\F}_2,\Gamma)$. 
The lefthand tensor factor in \cref{ssGalcl} can be identified as the $E_2$-page of the spectral sequence
\[H^\ast(G;\pi_\star E_{(\overline{\F}_2,\Gamma)})\implies \pi_{\star}E_{(\overline{\F}_2,\Gamma)}^{hC_{2^n}\times\Gal}
\cong \pi_{\star}E_h^{hC_{2^n}},\] which is isomorphic to the homotopy fixed point spectral sequence 
for $E_h$ by identifying $E_h\simeq E_{(\overline{\F}_2,\Gamma)}^{h\Gal}$. One can then verify that the 
morphism from \cref{ssmapLTs} on the $E_2$-page, sends $x\mapsto x\otimes 1$ and is therefore injective, 
implying that $u_V$ must be a $d_2$-cycle in the source of \cref{ssmapLTs}. 
Furthermore, the differentials in the target of \cref{ssmapLTs} are $W(\overline{\F}_2)$-linear 
differentials extended from those in the source. .

Now suppose that \cref{ssmapLTs} is injective on the $E_r$ page and that $x$ maps to zero in 
the $E_{r+1}$-page. This means that the image of $x$ must be hit by a $d_r$-differential. 
By $W(\overline{\F}_2)$-linearity of the differentials, we have that the image of $x$ can be 
written as a finite $W(\overline{\F}_2)$-linear combination of differentials from the source of 
\cref{ssmapLTs}. However, the image of $x$ must also be $\Gal$-invariant, forcing the linear 
combination to be trivial. Thus, $x$ itself must be the target of a $d_r$-differential, implying that 
\cref{ssmapLTs} is injective on $E_{r+1}$. 
Thus $u_V$ will be a $d_r$ cycle for any $r\geq 0$ in the $C_{2^n}\text{-}\mathrm{HFPSS}(E_h)$.
\end{proof}

Thus, our strategy for determining the self-duality of higher real $K$-theories associated 
to good Lubin-Tate spectra, is to show that certain 
orientation classes are permanent cycles in the $C_{2^n}\text{-}\mathrm{HFPSS}(E_{C_{2^n}}(m))$ 
or some more initial spectral sequence to deduce periodicities via \cref{oneuv}.

%\begin{rmk}
%We end by remarking that it should also be possible to avoid going through the Real bordism models, 
%$E^\R_h$, all together by appropriately adjusting the choice of generators in $\pi_{\ast\rho_2}\BPCn$ 
%according to the choice of Lubin-Tate spectrum. The author thanks Mike Hill for pointing this out. 
%\end{rmk}

%% file: Section5-IntShifts.tex
% !TEX root = master.tex

\subsection{A repository of periodicities for Lubin-Tate spectra}\label{computations}
We start with an important periodicity present in any Lubin-Tate spectrum which does not come 
from an orientation class. 

\begin{prop}\label{rhoperiodicity}
Let $E_h$ be a Lubin-Tate spectrum, and $G\leq\estab_h$ a finite subgroup containing $C_2$. 
Then $E_h$ is $\rho_G$-periodic. 
\end{prop}
\begin{proof}
In \cite{realorientation}, it is shown that a unit $u\in\pi_2 E_h$ has a $C_2$-equivariant lift 
to a unit in $\pi_{\rho_2}^{C_2} E_h$. The norm of this lift gives a unit in $\pi_{\rho_G}^G E_h$ 
which induces a $\rho_G$-periodicity by multiplication. 
\end{proof}

The next periodicity follows from the arguments in 
Section 9 of \cite{hhr16}. We include a short proof for completeness.

\begin{prop}\label{sigpc}
The class $u_{2\sigma}^{2^m}$ is a permanent cycle in the 
$C_{2^n}$-slice spectral sequence of $N(\tgen_m)^{-1}\BPCn$ . 
\end{prop}
\begin{proof}
From the $2$-typical case of Corollary 9.13 in \cite{hhr16} we have that 
$N(\tgen_m)u_{2\sigma}^{2^m}$ is a permanent cycle in 
$C_4\textrm{-}\mathrm{SliceSS}(\BPCfour)$. This class must then map to a permanent
cycle under the map of spectral sequences induced by the localization 
\[C_4\textrm{-}\mathrm{SliceSS}(\BPCn)\to C_4\textrm{-}\mathrm{SliceSS}(N(\tgen_m)^{-1}
\BPCn).\] The result then follows from the fact that $N(\tgen_m)$ 
is an invertible permanent cycle in the target spectral sequence.
\end{proof}

\begin{cor}\label{sigp}
Let $E$ be a good Lubin-Tate spectrum at height $h=2^{n-1}m$. 
Then for $0\leq r \leq n$, and $\sigma_{2^r}$ the sign representation of $C_{2^r}$, 
the restriction $i^\ast_{C_{2^r}}E$ has a 
$(2^{2^{n-r}m+1}-2^{2^{n-r}m+1}\sigma_{2^r})$-periodicity:
\[\Sigma^{2^{2^{n-r}m+1}\sigma_{2^r}} i^\ast_{C_{2^r}}E_h \simeq 
\Sigma^{2^{2^{n-r}m+1}}i^\ast_{C_{2^r}}E_h.\]
\end{cor}

\begin{proof}
The unit of the norm-restriction adjuction gives a $C_{2^r}$-map 
\[\mathrm{BP}^{((C_{2^r}))}\to i^\ast_{C_{2^r}}\BPCn.\] By \cref{localizationfactorization},
the composite \[\mathrm{BP}^{((C_{2^r}))}\to i^\ast_{C_{2^r}}\BPCn\to i^\ast_{C_{2^r}}E_h\] 
factors through $N(\tgen_{2^{n-r}m})^{-1}\mathrm{BP}^{((C_{2^r}))}$. 
The result then follows directly from \cref{sigpc}. 
\end{proof}

\begin{defn}
We refer to the periodicity of \cref{sigp} as the \emph{$\sigma_r$-periodicity} of $E_h$. 
\end{defn}

From this result we can also deduce an integer periodicity for $E_h$.
This is due to Hill-Hopkins-Ravenel. 
We include a proof here for completeness, but make no claim of 
originality for this statement or its proof. Furthermore, we note that the proof 
is essentially the same as the proof of Theorem 2.15 in \cite{hishiwax} for the case 
of $h=4$ and $C_4$. Roughly, the idea is to take the $\sigma_r$-periodicites, norm them 
up to the appropriate level, and take a suitable linear combination of them together
with the $\rho$-periodicity of \cref{rhoperiodicity}.

\begin{thm}[Hill-Hopkins-Ravenel]\label{intperiodicity}
Let $h = 2^{n-1}m$ and let \[p_h(C_{2^n}) =2^{h+n+1}.\] 
The spectrum $E_h^{hC_{2^n}}$ is $p_h(C_{2^n})$-periodic.  
Furthermore, letting \[p_h(C_{2^{n-\ell}}) = \frac{p_h(C_{2^n})}{2^{\ell}}\quad \text{for } 0\leq \ell \leq n-1,\] 
we have that $E_h^{hC_{2^{n-\ell}}}$ is $p_h(C_{2^{n-\ell}})$-periodic.
\end{thm}

\begin{proof}
Let $h = 2^{n-1}$. We prove the case of a general $\ell$ satisfying $0\leq\ell\leq n-1$. 
As a notationally reminder, we will be considering the sign representation of various subgroups 
of $C_{2^{n-\ell}}$ and so we write $\sigma_{2^r}$ for the sign representation of $C_{2^r}$. 
The only rotational representations appearing will be $C_{2^{n-\ell}}$-representations, so 
$\lambda_r$ denotes the $2$-dimensional representation given by rotation by $\frac{2\pi}{2^{n-\ell-r}}$

By \cref{localizationfactorization}, the map $\text{BP}^{((C_{2^{n-\ell}}))}\to E_h$ factors as in the diagram below
\begin{center}
\begin{tikzcd}
\text{BP}^{((C_{2^{n-\ell}}))}\ar[r]\ar[d]& i^\ast_{C_{2^{n-\ell}}}E_h\\
(N(\tgen^{C_2}_h)N(\bar{t}^{C_4}_{\frac{h}{2}})\cdots N(\tgen^{C_{2^{n-\ell}}}_{\frac{h}{2^{n-\ell-1}}}))^{-1}\text{BP}^{((C_{2^{n-\ell}}))}\ar[ur]
\end{tikzcd}.
\end{center}

Consider $r$ such that $\ell < r \leq n-1$. Since $N(\tgen^{C_{2^{n-r}}}_{2^r})$ is inverted, 
$u_{2\sigma_{2^{n-r}}}^{2^{2^r m}}$ is a permanent cycle in the $C_{2^{n-r}}$-HFPSS$(E_h)$
for $0\leq r\leq n-1$. 
Thus, the norm 
\[N_{2^{n-r}}^{2^{n-\ell}}(u_{2\sigma_{2^{n-r}}}^{2^{2^r m}}) = 
\frac{u_{2^{r-\ell-1}\lambda_{n-r-1}}^{2^{2^r m+1}}}{u_{2\sigma_{2^{n-\ell}}}^{2^{2^r m}}}\] is a permanent cycle in the 
$C_{2^{n-\ell}}$-HFPSS$(E_h)$.
It follows that $u_{2^{r-\ell-1}\lambda_{n-r-1}}^{2^{2^r m+1}}$
is a permanent cycle, leading to a $(2^{2^r m+r-\ell+1}-(2^{2^r m+r-\ell})\lambda_{n-r-1})$-periodicity.

Combining these periodicities with the $\rho_{2^{n-\ell}}$-periodicity of \cref{rhoperiodicity} yields
 a periodicity of 
\begin{equation*}
\begin{split}
2^{2^{n-1}m+1}\rho_{2^{n-\ell}} & 
+\sum_{r=\ell}^{n-1}2^{2^r m+\ldots+ 2^{n-2}m}(2^{2^r m +r-\ell+1}-2^{2^r m+r-\ell}\lambda_{n-r-1}) \\
&=2^{2^{n-1}m+1} + \sum_{r=\ell}^{n-1}2^{2^{n-1}m +r-\ell+1}\\
&= 2^{2^{n-1}m+n-\ell+1}. \qedhere
\end{split}
\end{equation*}

\end{proof}

\begin{rmk}
One can verify that at height $h=2$ this is the minimal periodicity. In \cite{picardHLS} it is also 
shown that this gives the minimal periodicity for $E_h^{hC_2}$. In all other cases it 
is unknown to the author whether this is true. 
\end{rmk}

For the next two periodicities we will make use of the Tate spectral sequence. 
In particular, we will need the following relationship between it and the homotopy fixed point 
spectral sequence. See, for example \cite[Theorem~3.6]{vanishinglines}.

\begin{thm} \label{tateisoregion}
For any $X\in\spectra^G$, the morphism 
\[\mathrm{HFPSS}(X)\to\mathrm{TateSS}(X)\] of $\mathrm{RO}(G)$-graded 
spectral sequences is an isomorphism on the $E_2$-page for classes in 
filtration $s > 0$ and a surjection for classes in filtration $s = 0$. 
Furthermore, there is a one-to-one correspondence between differentials 
with source in nonengative filtration.
\end{thm}

\noindent The strategy is then to prove that the $u_V$ in question is a permanent cycle 
in the Tate spectral sequence. Since these classes are in filtration zero, the theorem then 
implies these will be permanent cycles in the HFPSS. What makes computing with 
the Tate spectral sequence more accessbible is the following fact.

\begin{prop}
The Euler class $a_\lambda$ is invertible in the $C_4$-Tate spectral sequence of any $C_4$-spectrum.
\end{prop}
\begin{proof}
Since the fixed point set $\lambda^H$ is trivial for any nontrivial $H\leq C_4$, 
we can take $S(\infty \lambda)$ as a model for $EC_4$ and so $S^{\infty\lambda}$ is 
a model for $\widetilde{EC}_4$. Thus, for any $C_4$-spectrum $X$, the Tate tower 
\[\widetilde{EC}_4\smsh F((EC_4)_+, P^\bullet X)\simeq S^{\infty\lambda} \smsh F((EC_4)_+,P^\bullet X)
\simeq a_\lambda^{-1} F((EC_4)_+,P^\bullet X)\] consists of spectra on which $a_\lambda$ acts invertibly.
\end{proof}

\begin{prop}\label{height2pc}
The class $u_{2\sigma}u_{4\lambda}$ survives the 
$C_4$-HFPSS$(E_{C_4}(1))$. 
\end{prop}
\begin{proof}
    From Corollary 3.20 in \cite{hishiwax}, we have that 
    $\overline{\kappa} = N(\tgen_1)^6 u_{4\lambda}u_{6\sigma}a_{2\lambda}$ 
    is a permanent cycle in $C_4\text{-SliceSS}(\bpcfourone)$ and so also 
    in $C_4\text{-HFPSS}(\bpcfourone)$. From the composition 
    \begin{equation*}
    \text{HFPSS}(\bpcfourone)\to \mathrm{HFPSS}(E_{C_4}(1))\to \mathrm{TateSS}(E_{C_4}(1))
    \end{equation*}
    of morphisms of spectral sequences, we have that $\overline{\kappa}$ 
    is also a permanent cycle in the middle and rightmost spectral sequences. 
    Now by \cref{sigpc}, we know that $u_{4\sigma}$ is a permanent cycle 
    in the middle and rightmost spectral sequences. 
    It follows that 
    \[u_{2\sigma}u_{4\lambda}=
    \overline{\kappa}N(\tgen_1)^{-6}u_{4\sigma}^{-1}a_{2\lambda}^{-1}\] 
    is also a permanent cycle in the rightmost spectral sequence. 
    By \cref{tateisoregion} we have that 
    $u_{2\sigma}u_{4\lambda}$ is also a permanent cycle in the middle 
    spectral sequence.
\end{proof}

\begin{rmk}
The reader may compare the argument to that in the proof of \cite[Proposition~5.25]{bbhs}, 
where the authors work directly with the $C_4$-homotopy fixed point spectral sequence of the
Lubin-Tate spectrum, $E^{\R}_2$. 

One can also verify that the arguments of \cite{hhr17} show that $u_{2\sigma}u_{4\lambda}$ 
is a permanent cycle in the $\slicess(\BPCfour\langle 1\rangle)$. Thus, a more direct proof 
of \cref{height2pc} is to verify this and use the morphism 
\[\slicess(\BPCfour\langle 1\rangle)\to\hfpss(E_{C_{4}}(1)).\]
We opted to use the proof above as we found it easier to verify from the literature. 
\end{rmk}

\begin{rmk}\label{liftingargh2}
The reader uncomfortable with our use of \cref{e2rho} may prefer to argue as follows. 
In \cite{hishiwax} it is shown that $u_{2\lambda}$ and 
$u_{2\sigma}u_{4\lambda}$ support the differentials 
\[d_5 u_{2\lambda} = N(\tgen_1)u_\lambda a_{2\lambda}a_\sigma \quad\text{and}\quad 
d_{13}(u_{2\sigma}u_{4\lambda}) = N(\tgen_2)u_{4\sigma}u_\lambda a_{6\lambda}a_\sigma\] in the 
$\slicess(\BPCfour\langle 2\rangle)$. 
By explicit computation, one can show that the same differentials occur in the 
$\slicess(\BPCfour)$. Passing to $N(\tgen_1)^{-1}\BPCfour$, the first differential 
on $u_{2\lambda}$ will imply that $u_\lambda a_{2\lambda}a_\sigma = 0$ after $E_5$. The 
target of the $d_{13}$ must be killed on the $E_5$-page, so $u_{2\sigma}u_{4\lambda}$ is a 
$d_{13}$-cycle. The vanishing line from \cite{vanishinglines} then completes the proof.  
\end{rmk}

\begin{cor}\label{height2p}
Let $E_2$ be a good Lubin-Tate spectrum at height $h=2$. 
Then for $C_4\leq\mathcal{O}^\times_2$, $E_2$ has a $(10- 2\sigma - 4\lambda)$-periodicity:
\[\Sigma^{2\sigma+4\lambda}E_2\simeq_{C_4} \Sigma^{10}E_2.\]
\end{cor}

\begin{prop}\label{height4pc}
The class $u_{4\sigma}u_{16\lambda}$ survives the
$C_4\text{-}\hfpss(E_{C_4}(2))$.
\end{prop}
\begin{proof}
    From the Hill-Shi-Wang-Xu \cite{hishiwax} differential 
    \[d_{61}(u_{32\lambda}u_{4\sigma}a_\sigma) = N(\tgen_2)^5 u_{16\lambda}u_{20\sigma}a_{31\lambda} =: y\] 
    in the $\text{SliceSS}(\bpcfourtwo)$ we have that 
    $y$ is a permanent cycle in this spectral sequence and so also in 
    $\text{HFPSS}(\bpcfourtwo)$. From the composition 
    \begin{equation*}
    \text{HFPSS}(\bpcfourtwo)\to 
    \text{HFPSS}(E_{C_4}(2))  \to\text{TateSS}(E_{C_4}(2))
    \end{equation*}
    we have that $y$ is also a permanent cycle in 
    both the middle and rightmost spectral sequences. From \cref{sigpc} 
    we have that $u_{8\sigma}$ is a permanent cycle in the middle and rightmost spectral  
    sequences as well. It follows that 
    \[u_{4\sigma}u_{16\lambda} = N(\tgen_2)^{-5}u_{16\sigma}^{-1}a_{31\lambda}^{-1}y\] is a permanent cycle in the rightmost spectral sequence. By \cref{tateisoregion} 
    we have that $u_{4\sigma}u_{16\lambda}$ is also a permanent 
    cycle in the middle spectral sequence. 
    
\end{proof}

\begin{rmk}\label{liftingargh4}
As in \cref{liftingargh2}, the uneasy reader could check that the $d_{61}$ used in the 
theorem above lifts to a differential in $\BPCfour$. The same argument will hold after 
passing to the homotopy fixed point and tate fixed point spectral sequences for 
$N(\tgen_2)^{-1}\BPCfour$. 
\end{rmk}

\begin{cor}\label{height4p}
Let $E_4$ be a good Lubin-Tate spectrum at height $h=4$. 
Then for $C_4\leq\mathcal{O}^\times_4$, $E_4$ has a $(36- 4\sigma - 16\lambda)$-periodicity:
\[\Sigma^{4\sigma+16\lambda}E_4\simeq_{C_4} \Sigma^{36}E_4.\]
\end{cor}

\begin{prop}\label{uW1pc}
Let $E_4$ be a good Lubin-Tate spectrum at height $h=4$. 
The class $u_{12\sigma_8}u_{12\lambda_1}u_{32\lambda_0}$ is a permanent cycle in the 
$C_8$-homotopy fixed point spectral sequence for $E_4$. Thus, the $C_8$-spectrum $E_4$ has 
a $(100-12\sigma_8-12\lambda_1-32\lambda_0)$-periodicity. 
\end{prop}
\begin{proof}
    From \cref{height4pc} we have that $u_{4\sigma_4} u_{16\lambda}$ is a permanent cycle in the
    $C_4$-HFPSS$(E_4)$. From \cref{sigpc} we have $u_{8\sigma_4}$ is a permanent cycle there as well so that 
    the product $u_{12\sigma_4}u_{16\lambda}$ is a permanent cycle. It follows that the norm 
    \[N(u_{12\sigma_4} u_{16\lambda}) 
    = \frac{u_{12\lambda_1} u_{32\lambda_0}}{u_{44\sigma_8}}\] 
    is a permanent cycle in the $C_8$-HFPSS$(E)$.
    By \cref{sigpc}, $u_{4\sigma_8}$ is a permanent cycle as well. Thus the product
    \[u_{56\sigma_8}N(u_{12\sigma_4} u_{16\lambda}) = u_{12\sigma_8}u_{12\lambda_1}u_{32\lambda_0}\] 
    is a permanent cycle.
\end{proof}

\subsection{The integer shifts}
Finally, we employ the repository of periodicities in the preceding 
section to prove the self-duality of some higher real $K$-theories.

\begin{thm}\label{C4shifts}
Let $E$ be a good Lubin-Tate spectrum at height $h$. Then 
\[DE\simeq_{C_4}\begin{cases} \Sigma^{12}E & h=2\\
\Sigma^{-h^2}E & h=4,8\end{cases}.\] 
\end{thm}

\begin{proof}
We start with the $h=2$ case. By \cref{height2pc}, $J(10-2\sigma-4\lambda) = 0$. 
We can rewrite this representation as 
\[10-2\sigma-4\lambda = 12 + (2 + 2\sigma) - 4\rho_4.\] 
\cref{rhoperiodicity} implies $J(\rho_4) = 0$. By \cref{thec4rep}, $D(E^{hC_4})\simeq 
J(-2-2\sigma)$. Thus $DE\simeq_{C_4}\Sigma^{12}E$ in the case $h=2$.

Now for $h=4$. Applying \cref{sigp}, with $n=m=r=2$ gives the $\sigma_4$-periodicity 
$J(8 - 8\sigma) = 0$. By
\cref{thec4rep}, the dual is given by 
\[J(-8-8\sigma) = -16.\] Thus $DE\simeq_{C_4}\Sigma^{-16}E$ in the case $h=4$.
The $h=8$ result follows similarly from applying \cref{sigp} with $n=r=2, m=4$.  
\end{proof}

\begin{cor}\label{C4shiftshfpcor}
    Let $E$ be a good Lubin-Tate spectrum at height $h$ and $\estab_h$ its extended stabilizer group. 
    If $G\leq\estab_h$ is a finite subgroup such that $G\cap\mathcal{O}^\times_h\leq C_4$.
    Then \[DE^{hG}\simeq\begin{cases} \Sigma^{12}E^{hG} & h=2\\
    \Sigma^{-h^2}E^{hG} & h=4,8\end{cases}.\] 
\end{cor}
\begin{proof}
Let $\widetilde{\Gal}$ be the image of $G$ under $\estab_h\to\Gal(k/\F_2)$, where 
$k = E_0/\mathfrak{m}$. The kernel of $G\to\widetilde{\Gal}$ is then $G_0 = G\cap\mathcal{O}^\times_h$. 
Then from \cite[Lemma~1.37]{bobkovagoerss}, we have 
\[E^{hG_0}\simeq \widetilde{\Gal}_+\smsh E^{hG}.\] 
If $DE^{hC_4}\simeq\Sigma^s E^{hC_4}$, then by restriction $DE^{hG_0}\simeq\Sigma^s E^{hG_0}$. 
Thus 
\begin{equation*}
\begin{split}
DE^{hG} &\simeq (DE^{hG_0})^{h\widetilde{\Gal}}\\
&\simeq (\Sigma^s E^{hG_0})^{h\widetilde{\Gal}}\\
&\simeq (\Gal_+\smsh\Sigma^s E^{hG})^{h\widetilde{\Gal}}\simeq \Sigma^s E^{hG}.
\end{split}
\end{equation*}
\end{proof}

\begin{rmk}
We note that the result above for heights $h = 4,8$ was obtained only from knowing the 
$\sigma$-periodicity. One might wonder whether this strategy could be generalized 
to higher heights.  Unfortunately, this is not the case.
At height $h = 2^{n-1}$, the element $N(\bar{t}_{2^{n-2}})$ is inverted
leading to a $(2^{2^{n-2}+1} - 2^{2^{n-2}+1}\sigma)$-periodicity, while
the dualizing $C_4$-representation is given in this case by 
\[V_{2^{n-1}}(C_4) = 2^{2n-3}+2^{2n-3}\sigma.\] It appears to be a lucky coincidence that 
for $n = 3,4$, the number of copies of $\sigma$ in the $\sigma$-periodicity is equal to 
the number of copies of $\sigma$ in the dualizing representation, as this will not be the 
case for any $n>4$.
It should be noted that this does not imply that the dual in this case
is not an integer shift. In fact, one can see this at the height $h=2$.
The $\sigma$-periodicity there is $4-4\sigma$, and the dualizing 
representation is $2+2\sigma$, yet $DE_2^{hC_4}\simeq \Sigma^{12}E_2^{hC_4}$.
\end{rmk}

Following the proof of the equivalence $DE_2^{hC_4}\simeq \Sigma^{12}E_2^{hC_4}$ in the preceding theorem, 
we can find a $C_8$-equivariant integer shift.
The key in that proof was finding an $\mathrm{RO}(G)$-periodicity of 
dimension zero of the form
\begin{equation*}\tag{4.3}\label{form}
\textrm{(shift)}+\textrm{(dualizing representation)} - \textrm{(copies of regular representation)}
\end{equation*}
This is precisely our approach in the proof of \cref{C8shift} below.

\begin{thm}\label{C8shift}
Let $E_4$ be a good Lubin-Tate spectrum at height $h=4$. Then 
\[DE_4\simeq_{C_8} \Sigma^{112}E_4.\]

\begin{proof}
Suppose $DE_4^{hC_8}\simeq \Sigma^{s_D}E_4^{hC_8}$ for some integer 
$s_D$. From \cref{C4shifts},
$D(E_4^{hC_4})\simeq\Sigma^{-16}E_4^{hC_4}$ and from 
\cref{intperiodicity}, $E_4^{hC_4}$ is $128$-periodic. We then have 
\[s_D\equiv -16 \mod 128.\] 
Consider a virtual representation of the form in \cref{form}.
\[W_\ell = s_D + V_4 - \ell\rho_8.\] In order for $\textrm{dim}W_\ell = 0$, we
need $s_D = 8\ell - 16$.
The first such representation to consider at $\ell = 16$ is
\[W_1 = 112 + 4(1+\sigma+\lambda_1) - 16\rho_8, \]
where, recall the dualizing $C_8$-representation from \cref{thedualizingrep} is 
\[V_4 = 4(1+\sigma+\lambda_1),\] and that \[\rho_8 = 1 + \sigma_8 + \lambda_1 + 2\lambda_0.\]
There is a corresponding orientation class 
\begin{equation*}\label{C8pc}
u(W_1) = u_{12\sigma_8} u_{12\lambda_1} u_{32\lambda_0}, 
\end{equation*}
such that if $u(W_1)$ survives the $C_8$-HFPSS$(E_4)$, then 
$J(W_1) = 0$, implying that $J(-V_4) = 112$. 
The class $u(W_1)$ is precisely the class from \cref{uW1pc}, so we are done.
\end{proof}
\end{thm}

\begin{cor}\label{C8shiftshfpcor}
Let $E_4$ be a good Lubin-Tate spectrum at height $h=4$ with extended stabilizer group 
$\estab_4$. Suppose $G\leq\estab_4$ is such that $G\cap\mathcal{O}^\times_4\leq C_8$. Then 
\[DE_4^{hG}\simeq \Sigma^{112}E_4^{hG}.\]
\end{cor}
\begin{proof}
This follows as in the proof of \cref{C4shiftshfpcor}. 
\end{proof}

\subsection{Conclusions}

We end with some remarks on extending the results in this paper, particularly 
for the group $C_4$. 
The main observation is that if 
\[\tag{4.4}\label{shifteq}DE_h^{hC_{2^n}}\simeq\Sigma^{s_h(C_{2^n})}E_h^{hC_{2^n}},\] for some $s_h(C_{2^n})\in\Z$, 
then, by restriction, the same dualizing shift must hold for $E_h^{hC_{2^r}}$, for any $1\leq r \leq n$. 
Since at all heights, $DE_h^{hC_2}\simeq \Sigma^{-h^2}E_h^{hC_2}$ \cite{dualizingspheres}, 
we must have \[\tag{4.5}\label{shiftcondition}s_h(C_{2^n})\equiv -h^2 \mod p_h(C_2),\] where $p_h(C_2)=2^{h+2}$ is the 
integer periodicity from \cref{intperiodicity}. Furthermore, the second part of \cref{intperiodicity} implies
that \cref{shiftcondition}
strongly limits the possible values of $s_h$. For example, if \cref{shifteq} holds, then 
\[s_h(C_4) = -h^2 \text{ or} -h^2 + p_h(C_2).\] 
%Assuming a strong converse to \cref{oneuv} holds, 
%the slice differentials theorem of HHR \cite{hhr16} would imply that the answer 
%cannot be the former. 

%\begin{quest}\label{endq1}
%Suppose $u_{k\sigma_{2^n}}$ supports a differential in the slice spectral
%sequence of  $E_{C_{2^n}}(m)$ and let $h=2^{n-1}m$. 
%Does this imply that $E_h$ is not $(k-k\sigma_{2^n})$-periodic? 
%\end{quest}

%Given the proof of \cref{intperiodicity}, the following question seems related and interesting 
%in and of itself. 

%\begin{quest}\label{endq2}
%Is $p_h(C_{2^n})$ the smallest periodcity of $E_h^{hC_{2^n}}$?
%\end{quest}

It is reasonable to expect $E_{C_{2^n}}(m)$ to be $C_{2^n}$-equivariantly equivalent to $E_h$ after 
a suitable completion. Indeed, in \cite[Proposition~7.6]{lubintatemodels}, the authors show that 
after $\mK(h)$-localizing, the map underlying \cref{factorization} is a finite Galois extension. 
This, together with the slice differentials theorem of \cite{hhr16}, suggests that $s_h(C_4)\neq -h^2$.   

\begin{quest}\label{endq3}
Let $E_h$ be a Lubin-Tate spectrum at an even height $h$ viewed as a $C_4$-spectrum.
For which $h$ is $E_h$ $(V_h(C_4) + 2^{h+2})$-periodic?
\end{quest}

%% file: AppendixA-nonequivariantdual.tex
% !TEX root = master.tex

\section{The dual of Lubin-Tate spectra nonequivariantly}\label{nonequivariatnappendix}

In this appendix we review the argument from \cite{StricklandGHduality} establishing the 
self-duality of Morava $E$-theory nonequivariantly.
Along the way we will prove a well-known result going back to Devinatz-Hopkins 
\cite{DevinatzHopkins99} establishing a useful spectral sequence for studying $\mK(h)$-local duality.
Our contribution is an exposition encompassing a general class of $\mK(h)$-local 
pro-Galois extensions, as well as the description of the homotopy groups of the dual of a Lubin-Tate 
spectrum associated to an algebraically closed field in \cref{hmtpyClosedEthy}, which to the 
knowledge of the author does not appear elsewhere in print. 
This appendix grew out of an attempt to use $E_h^{\R}$ as in \cref{rblt} as our model of 
$E$-theory throughout the paper. However, we conclude from \cref{hmtpyClosedEthy} that
the duality results in the main text do not apply to such Lubin-Tate theories directly. 

We begin with some generalities on $\mK(h)$-local pro-Galois extensions.
Let $A\to B$ be such a Galois extension with profinite Galois group $\cG = 
\lim_i \cG_i$ and let $A\to B_i$ be finite $\cG_i$-Galois extensions so that 
\[B\simeq L_{\mK(h)}\colim_i B_i.\] We work in $\spectra_{\mK(h)}$ where colimits and 
smash products must be $\mK(h)$-localized. The latter of these will be left notationally implicit. 
We begin by recalling some useful definitions and notation. 

\begin{sdefn}
    Let $B\in\spectra_{\mK(h)}$ and $S = \lim_i S_i$ a profinite set with $S_i$ finite. 
    The \emph{$\mK(h)$-local continous mapping spectrum} from $S$ to $B$ is defined as 
    \[F_c(S_+,B) = L_{\mK(h)}\colim_i F( (S_i)_+,B).\]
\end{sdefn}
    
\begin{sdefn}
    Let $A = \lim_i A_i$ and $S = \lim_i S_i$ be sets given as inverse limits. 
    The \emph{continuous mapping set} from $S$ to $A$ is defined as 
    \[\map_c(S_+,A) = \lim_i\colim_j \map( (S_j)_+,A_i).\]
\end{sdefn}

\begin{slem}\label{thegaloislemma}
    We have the following equivalences: 
    \begin{enumerate}[(a)]
    \item $B^{\smsh_A s} \simeq F_c(\cG^{s-1}_+,B)$ for any $s \geq 1$, and
    \item $F_A(B_i,B)\simeq B [\cG_i]$. 
    \end{enumerate}
    \noindent Furthermore, the equivalence in (b) can be taken to be equivariant 
    with respect to the $\cG$-action on the target on the left and the diagonal action on the right.
\end{slem}
\begin{proof}[Proof of~(a)]
    
    The case $s = 1$ is trivial and $s = 2$ is Equation 8.1.2 of \cite{rognesgalois}. 
    We then proceed by induction. Suppose $B^{\ox s-1}\simeq F_c(\cG^{s-2},B)$. 
    Then 
    \begin{equation*}
    \begin{split}
    B^{\smsh_A s} &\simeq F_c(\cG^{s-2},B)\smsh_A B\\
    &\simeq L_{\mK(h)}( \colim_i F(\cG_i^{s-2},B)\smsh_A B)\\
    &\simeq L_{\mK(h)}( \colim_i F(\cG_i^{s-2},B\smsh_A B))\\
    &\simeq L_{\mK(h)}(\colim_i F(\cG_i^{s-2},F_c(\cG,B)) )\\
    &\simeq L_{\mK(h)}(\colim_i\colim_j F(\cG_i^{s-2}\times \cG_j,B))\\
    &\simeq F_c(\cG^{s-1},B). \qedhere
    \end{split}
    \end{equation*}
\end{proof}
\begin{proof}[Proof of~(b)]
    This follows directly from \cite[Lemma~6.1.2(c)]{rognesgalois} where an explicit equivalence 
    is given as the adjoint of the composite map 
    \[(B_i\smsh(\cG_i)_+)\smsh_A B\to B_i\smsh_A B\to B\smsh_A B\to B. \qedhere\]
\end{proof} 

\begin{rmk}
We choose generalized Moore spectra $\{M_I\}$ as in \cite[Section~4]{hoveystrickland} so that 
$L_{\mK(h)} X = \lim_I L_h X\smsh M_I$, where $L_h(-)$ denotes localization at height $h$ Morava $E$-theory.
\end{rmk}

\begin{sdefn}
    We say $B$ \emph{satisfies ML} if for all $t\in\Z$, the inverse system of abelian groups,
    $\{\pi_t (B\smsh M_I)\}_I$, satisfies the Mittag-Leffler condition. 
\end{sdefn}

\begin{sthm}\label{thessthm}
    Let $A\to B$ be a $\mK(h)$-local pro-$\cG$-Galois extension which admits descent in the sense of 
    \cite[Definition~3.18]{mathewgalois} and suppose that $B$ satisfies ML. 
    There are conditionally convergent spectral 
    sequences of the form
    
    \begin{equation}\label{regss}
        E_2^{s,t}\cong H^s_{c}(\cG;B_t)\implies \pi_{t-s}A
    \end{equation}
    \noindent and 
    \begin{equation}\label{dualss}
    E_2^{s,t}\cong H^s_{c}(\cG;B_t\llbracket\cG\rrbracket)\implies \pi_{t-s}D_A B.
    \end{equation}
\end{sthm}

\begin{proof}
    Since $A\to B$ admits descent, by 
    \cite[Proposition~3.22]{mathewgalois} we have an equivalence 
    \[A \simeq \text{Tot} B^{\smsh_A\bullet + 1}\] between $A$ and the totalization of the 
    $\mK(h)$-local Amitsur complex of $A\to B$. 
    
    Since $A\to B$ is Galois, by part (a) of \cref{thegaloislemma} we have an equivalence 
    \[B^{\smsh_A s}\simeq F_c(\cG^{s-1}, B) = L_{\mK(h)}\colim_i F(\cG_i^{s-1},B). \]
    Applying $\pi_t(-)$ to the tower 
    \[\cdots\to\text{Tot}_2 B^{\smsh_A\bullet +1}\to\text{Tot}_1 B^{\smsh_A\bullet +1}\to 
    \text{Tot}_0 B^{\smsh_A\bullet +1}\]
    then yields a spectral sequence of the form 
    \[\pi_tF_c(\cG, B)\implies \pi_tA.\] Using the assumption at $B$ is ML and that the
    Mittag-Leffler condition is stable under base-change \cite{emmanouil} we can identify the 
    $E_2$-page as in \cref{regss}.
    
    Now applying $\pi_t F_A(B,-)$ to the Tot-tower yields a spectral sequence of the form 
    \begin{equation}\label{ssform1}
    \pi_t F_A(B, F_c(\cG^{s-1},B))\implies \pi_{t-s} D_A B.
    \end{equation}
    We wish to identify this initial page as a group of continous cochains on $\cG$. 
    To this end, note that since $A\to B$ is a pro-$\cG$-Galois extension, 
    \[B \simeq L_{\mK(h)}\colim_{res,i} B_i,\] and each $A\to B_i$ is a finite Galois 
    extension (of Galois group $\cG_i$). It follows that  
    \begin{equation*}
    \begin{split}
    F_A(B,F_c(\cG^{s-1},B)) &\simeq F_A(L_{\mK(h)}\colim_{res,i} B_i,F_c(\cG^{s-1},B))\\
    &\simeq \lim_{D(res),i}F_A(B_i,F_c(\cG^{s-1},B)) \\
    &\simeq \lim_{D(res),i}L_{\mK(h)}\colim_j F_A(B_i,F(\cG_j^{s-1},B)) \\
    &\simeq \lim_{D(res),i}L_{\mK(h)}\colim_j F(\cG_j^{s-1},F_A(B_i,B)) \\
    & = \lim_{D(res),i}\lim_I \colim_j F(\cG_j^{s-1},B[\cG_i])\smsh M_I .
    \end{split}
    \end{equation*}
    
    \noindent We note that in the second equivalence above we use the fact that 
    for $\mK(h)$-local $Y$, $F(L_{\mK(h)}X, Y)\simeq F(X,Y)$.
    Now by \cite{emmanouil} since $\{\pi_t B\smsh M_I\}_I$ is ML, so is the base-changed system 
    \[\{\colim_j\map(\cG_j^{s-1},\Z[\cG_i])\ox_{\Z}\pi_t B\smsh M_I\}_I. \]
    Thus \[\pi_t\lim_I\colim_j F(\cG_j^{s-1},B[\cG_i])\smsh M_I\cong \lim_I\colim_j\map(\cG_j^{s-1},
    (\pi_t B\smsh M_I)[\cG_i]).\]
    Then note that for fixed $j,I$, and $i < i'$, the map 
    \[\map(\cG_j^{s-1},\pi_t(B\smsh M_I)[\cG_{i'}])\to\map(\cG_j^{s-1},\pi_t(B\smsh M_I)[\cG_i])\] is 
    just the map induced by the quotient $\cG_{i'}\to\cG_i$ and so 
    is surjective. The resulting map on colimits is surjective. It remains to check that the resulting map 
    after taking $\lim_I$ must also be surjective. This holds because the kernel of the map above 
    associated to $i< i'$ is given by 
    \[\map(\cG_j^{s-1},\pi_t(B\smsh M_I)[U_i/U_{i'}]).\] The resulting inverse system indexed on $I$ 
    is thus ML so that the $\lim^1_I$-term vanishes and the map on limit is surejctive.  
    This implies that the system 
    \[\{\lim_I\colim_j\map(\cG_j^{s-1},(\pi_t B\smsh M_I)[\cG_i])\}_i\] is ML and so 
    upon applying $\pi_t$, all $\lim^1$-terms vanish so that 
    \[\pi_t F_A(B,F_c(\cG^{s-1},B))\cong\map^c(\cG^{s-1},B^t\llbracket\cG\rrbracket)
    = \lim_{i,I}\colim_j\map(\cG_j^{s-1},(\pi_t B\smsh M_I)[\cG_i]).\]
    One then readily identifies the $d^1$-differential as the usual one so that 
    the spectral sequence \cref{ssform1} can be written as in \cref{dualss}.  
\end{proof}

\begin{srmk}
    We remark that the first spectral sequence in the theorem above was constructed
    in \cite[Main Theorem~B]{Li2023}. 
\end{srmk}
    
\begin{srmk}\label{modulermk}
    We emphasize that the coefficient $\cG$-module appearing in the $E_2$-page of the spectral 
    sequence constructed above is the topological $\cG$-module given by
    \[B_t\llbracket\cG\rrbracket = \lim_{i,I}\pi_t(B\smsh M_I)[\cG_i]\] with the 
    inverse limit topology and the diagonal $\cG$-action. Now fix a prime $p$. If each $\pi_t(B\ox M_I)$ is 
    a finite discrete $p$-torsion $\cG$-module, then we can identify this module as 
    the completed tensor product \[B_t\llbracket\cG\rrbracket \cong B_t\hat{\ox}_{\Z_p}\Z_p\llbracket\cG\rrbracket.\]
    In this case, via the usual shearing isomorphism, we can take the $\cG$-action on $B_t$ to be trivial and 
    that on $\Z_p\llbracket\cG\rrbracket$ the left multiplication. Then $B_t\llbracket\cG\rrbracket$ is 
    an induced module.
\end{srmk}

\begin{sdefn}
    We say $B$ is \emph{$\mK(h)$-locally profinite at $p$} if 
    for all $t, I$, $\pi_t(B\smsh M_I)$ is a finite discrete $p$-torsion $\cG$-module. 
\end{sdefn}

\begin{sthm}\label{thenonequivariantdualthm}
    Suppose $A\to B$ is a $\mK(h)$-local pro-$\cG$-Galois extension that admits descent, 
    and that $B$ is $\mK(h)$-locally profinite. Suppose additionally that
    $\cG$ is a virtually orientable Poincar\'e duality group at $p$ with $\text{vcd}_p\cG = d$, 
    then there is an equivalence of spectra \[D_A B\simeq \Sigma^{-d}B.\]
\end{sthm}
    
\begin{proof} 
    Suppose $\cH\leq_o\cG$ is an open normal subgroup that is an orientable Poincar\'e duality 
    group of dimension $d$. Then $A\to B$ can be thought of as the composition 
    \[A\to B^{h\cH}\to B\] where the first map is a finite Galois extension and the second is a 
    pro-$\cH$-Galois extension. We begin by dualizing the latter. 
    
    First we note that arguing as in the proof of \cite[Proposition~3.2.6]{Li2023}
    one can show that $B^{h\cH}\to B$ admits descent.
    Consider the spectral sequence of \cref{thessthm}
    \[H^s_{c}(\cH;B_t\llbracket\cH\rrbracket)\implies \pi_{t-s}D_{B^{h\cH}} B. \]
    From the assumption that $\cH$ is a Poincare 
    duality group at $p$ we can rewrite the $E_2$-page as
    \begin{equation*}
    \begin{split}
    H^s(\cH;B_t\llbracket\cH\rrbracket) 
    &\cong H_{d-s}(\cH;B_t\llbracket\cH\rrbracket)\\
    &\cong H_{d-s}(\cH;B_t\hat{\ox}_{\Z_p}\Z_p\llbracket\cH\rrbracket)\\
    &\cong H_{d-s}(1;B_t) \\
    &\cong \begin{cases} B_t & s = d\\ 0 & \text{else}\end{cases}.
    \end{split}
    \end{equation*} The first isomorphism  above is Poincar\'e duality (see, for example, 
    \cite[Proposition~4.5.1]{SymondsWeigel2000}).The second isomorphism just follows from the 
    definition of the completed tensor product (see \cref{modulermk}), and the third isomorphism is Shapiro's Lemma. 
    
    Choosing a unit $u\in\pi_{-d}D_{B^{h\cH}} B\cong\pi_0 B$ and using the $B$-module structure gives 
    an equivalence of spectra \[D_{B^{h\cH}}B\simeq \Sigma^{-d}B.\] 
    
    Now since $A\to B^{h\cH}$ is a finite Galois extension, we have 
    \[D_A B^{h\cH}\simeq B^{h\cH}. \] Putting all this together using an adjunction, we have 
    \[D_A B \simeq F_{B^{h\cH}}(B, D_A B^{h\cH})\simeq D_{B^{h\cH}}B\simeq \Sigma^{-d}B. \qedhere\]
\end{proof}

\begin{rmk}
We remark that one should also be able to deduce this theorem from a general 
version of Clausen's \cref{LinearizationHypothesis}. 
\end{rmk}

\begin{sex}\label{finitefieldLTtheories}
    Fix a $p$-typical formal group law $\Gamma$ of height $h$ defined over $\F_p$. Let 
    $k$ be an algebraic extension of $\F_p$ and let
    $\estab = \estab(k,\Gamma)$ and $E = E_{(k,\Gamma)}$ be 
    the associated (extended) Morava stabilizer group, and Lubin-Tate spectrum. Note that 
    $E$ has coefficients abstractly isomorphic to 
    \[E_\ast \cong W(k)\llbracket u_1,...,u_{h-1}\rrbracket[u^\pm],\] where $W(k)$ denotes the 
    ring of Witt vectors of $k$, $|u_i| = 0$, and $|u| = -2$. 
    
    Consider the $\mK(h)$-local pro-$\estab$-Galois extension 
    \[L_{\mK(h)}S\to E.\] Mathew in \cite[Proposition~10.10]{mathewgalois} shows that this extension 
    admits descent. Since $E_tM_I\cong E_t/I$ and the morphisms in this system are 
    the quotient maps, we have that $E$ is $\mK(h)$-locally ML and so there is a spectral sequence 
    \[H^s(\estab;E_t\llbracket\estab\rrbracket)\implies\pi_{t-s}DE.\]
    
    Furthermore, if we assume that $k$ is a finite extension of $\F_p$ then 
    $E_t/I$ is in fact finite so that $E$ is $K$-locally profinite in this case. 
    Lastly, if we assume that $\Gamma$ has all of its automorphisms over $k$. Then 
    $\estab$ is known to be virtually an orientable Poincar\'e duality group of dimension 
    $h^2$. Thus by \cref{thenonequivariantdualthm} we have that under these assumptions 
    \[DE\simeq \Sigma^{-h^2}E.\] This is originally a result of Hopkins, and the proof here is 
    essentially the same as Strickland's in
    \cite[Proposition~16]{StricklandGHduality}.
  
\end{sex}

When $k$ is not a finite extension, we need to work a little harder. Let $\Gal = \Gal(\overline{\F}_p/\F_{p^h})$
and for $m$ a multiple of $h$, $\Gal^m = \Gal(\F_{p^m}/\F_{p^h})$. 
By the normal basis theorem for Galois extensions of local fields 
\cite[Theorem~6.1]{ringGal}, we can identify $W(\F_{p^m})\cong\map(\Gal^m,W(\F_{p^h}))$. 
Letting $m$ range through multiples of $h$, we then have
\[W(\overline{\F}_p)=(\colim_m W(\F_{p^m}))_p^\wedge\cong
\map_c(\Gal,W(\F_{p^h})).\] 
One can then work out that the continuous $W(\F_{p^h})$-dual is given by 
\[D^c_hW(\overline{\F}_p) = \map_c(W(\overline{\F}_p),W(\F_{p^n}))
\cong W(\F_{p^h})\llbracket\Gal\rrbracket.\]
In the theorem below, we compute the homotopy groups of the dual of the algebraically closed Morava $E$-theory, 
$E(\overline{\F}_p)$.
Given the homotopy groups, we conclude that the duality results in the main text do not directluy apply to 
this spectrum. Those results would imply that $E(\overline{\F}_p)$ is self-dual up to a $-h^2$ suspension shift, 
and although we do observe this shift in the homotopy groups, 
it is accompanied by a levelwise continuous dualization of the coefficients.

%Here's a proof that this is the dual. 
%\begin{equation*}
%\begin{split}
%D^c_{W(\F_{p^h})}W(\overline{\F}_p) &= \lim_{\ell}\colim_k \Hom_{W(\F_{p^h})}(W(\overline{\F}_p)/p^k,W(\F_{p^h})/p^\ell)\\
%&= \lim_\ell\colim_k\lim_m\Hom_{W(\F_{p^h})}(W(\F_{p^h})/p^k[\Gal^m],W(\F_{p^h})/p^\ell)\quad\text{(continuity of Hom)}\\
%&\cong \lim_\ell\colim_k\lim_m\Hom_{W(\F_{p^h})}(W(\F_{p^h})/p^k[\Gal^m],
%W(\F_{p^h})/p^k\otimes_{W(\F_{p^h})}W(\F_{p^h})/p^\ell)\quad\text{(taking $\ell < k$)}\\
%&\cong\lim_\ell\colim_k\lim_m W(\F_{p^h})/p^k[\Gal^m]\otimes_{W(\F_{p^h})}W(\F_{p^h})/p^\ell\quad\text{(self duality of finite Galois)}\\
%&\cong\lim_{\ell}\colim_k\lim_m (W(\F_{p^h})/p^k\otimes_{W(\F_{p^h})}W(\F_{p^h})/p^\ell)\otimes_\Z \Z[\Gal^m]\\
%&\cong \lim_\ell\lim_m W(\F_{p^h})/p^\ell[\Gal^m]\quad\text{(again taking $\ell < k$)}.
%\end{split}
%\end{equation*}

\begin{sthm}\label{hmtpyClosedEthy}
    Let $\Gamma$ be the height $h$ Honda formal group law, and for $k$ an algebraic extension of 
    $\F_{p^h}$, let $E(k)$ be the Lubin-Tate spectrum associated to the pair $(k,\Gamma)$. 
    Then the homotopy groups of $DE(\overline{\F}_p)$ are given by 
    \[\pi_\ast DE(\overline{\F}_p)\cong E(\F_{p^h})_{\ast+h^2}\hat{\ox}_{W(\F_{p^h})}D^c_h W(\overline{\F}_p).\]
\end{sthm}

\begin{proof}
    The proof is based on the fact that $E(\F_{p^h})\to E(\overline{\F}_p)$ is a $\mK(h)$-local 
    pro-$\Gal$-Galois extension which is pro-trivial. This means in particular that 
    we have a $\Gal$-equivariant equivalence \cite[Lemma~1.37]{bobkovagoerss}
    \[E(\F_{p^m})\simeq F(\Gal^m_+, E(\F_{p^m})^{h\Gal^m})\simeq F(\Gal^m_+, E(\F_{p^h})).\] 
    We then have
    \begin{equation*}
    \begin{split}
        E(\overline{\F}_p) &\simeq L_{\mK(h)}\colim_m F(\Gal^m_+, E(\F_{p^h}))\\
        &\simeq \lim_I\colim_m F(\Gal^m_+,E(\F_{p^h})\smsh M_I).
    \end{split}
    \end{equation*}
    Now the inverse system $\{\colim_m\pi_tF(\Gal^m_+,E(\F_{p^h})\smsh M_I)\}_I$ consists of 
    surjective maps so we have \[E(\overline{\F}_p)_\ast\cong\map_{c}(\Gal,E(\F_{p^h})_\ast).\]
    
    Consider the spectral sequence 
    \[H^s(\Gal; E(\overline{\F}_p)_t\llbracket\Gal\rrbracket)\implies \pi_{t-s}D_{E(\F_{p^h})}E(\overline{\F}_p)\]
    of \cref{thessthm}. We can then identify the $\Gal$-module appearing in the $E_2$-page as
    \[E(\overline{\F}_p)_\ast\llbracket\Gal\rrbracket \cong\map_{c}(\Gal_+,E(\F_{p^h})_\ast\llbracket\Gal\rrbracket),\]
    where $\Gal$ acts by conjugation and the action on $E(\F_{p^h})_\ast\llbracket\Gal\rrbracket$ is 
    through the action on $\Gal$. We can shear this action onto the source making this a coinduced module 
    so that 
    \[H^s(\Gal;E(\overline{\F}_p)_t\llbracket\Gal\rrbracket)\cong H^s(1;E(\F_{p^h})_t\llbracket\Gal\rrbracket)
    \cong\begin{cases} E(\F_{p^h})_t\llbracket\Gal\rrbracket & s = 0 \\ 0 & \text{else} \end{cases}.\]
    %shearing detail
    %To be completely explicit, if $n | m$ we have a quotient map $\phi^m_n:\Gal^m\to\Gal^n$ and we can define 
    %a shearing isomorphism 
    %\[\map_c(\Gal^m_+,E(\F_{p^h})[\Gal^n])\to \map_c(\Gal^m_+,E(\F_{p^h})[\Gal^n])\]
    %\[f\mapsto sh(f),\] where $sh(f)(g) = \phi^m_n(g^{-1})\cdot f(g)$. This is an isomorphism 
    %of $\Gal^m$-modules where the source is given the diagonal action $f\mapsto \phi^m_n(h)fh^{-1}$ and the target 
    %the action $f\mapsto fh^{-1}$. Our map above is the limit-colimit of this isomorphism hence is also an isomoprhism 
    %as desired.
    Thus, we have 
    \[\pi_\ast D_{E(\F_{p^h})}E(\overline{\F}_p)\cong E(\F_{p^h})_\ast\llbracket\Gal\rrbracket\cong 
    E(\F_{p^h})_\ast\hat{\ox}_{W(\F_{p^h})}D_h^c W(\overline{\F}_p).\] 
    Using an adjuction, we have 
    \[DE(\overline{\F}_p)\simeq F_{E(\F_{p^h})}(E(\overline{\F}_p),DE(\F_{p^h}))
    \simeq \Sigma^{-h^2}D_{E(\F_{p^h})}E(\overline{\F}_p).\]
    The result follows.
\end{proof}